\documentclass[12pt]{amsart}
\usepackage{amsmath}
\usepackage{amssymb}
\usepackage{verbatim}
\usepackage{graphicx}
\usepackage{color}
\usepackage{dsfont}
\usepackage{hyperref}
\hypersetup{colorlinks=true, linkcolor=blue, citecolor=magenta, urlcolor=wine}
\usepackage{url}
\usepackage{fancyhdr}
\usepackage{tikz-cd}
\usepackage{tikz}
\usetikzlibrary{shapes.geometric,arrows.meta}
\tikzcdset{arrow style=tikz}
\usepackage{calc}
\usepackage{stmaryrd}

\usepackage[new]{old-arrows}  

\tikzstyle{wc}=[fill=white, draw=black, shape=circle]
\tikzstyle{ws}=[fill=white, draw=black, shape=rectangle]
\tikzstyle{bd}=[fill=black, draw=black, shape=circle, inner sep=0pt, minimum size=4pt]
\tikzstyle{wd}=[fill=white, draw=black, shape=circle, inner sep=0pt, minimum size=4pt]
\tikzstyle{bt-up}=[fill=black, draw=black, shape=isosceles triangle, inner sep=0pt, minimum size=4pt, shape border rotate=90]
\tikzstyle{bt-down}=[fill=black, draw=black, shape=isosceles triangle, inner sep=0pt, minimum size=4pt, shape border rotate=-90]
\tikzstyle{wt-up}=[fill=white, draw=black, shape=isosceles triangle, shape border rotate=90, inner sep=0pt, minimum size=4pt]
\tikzstyle{wt-down}=[fill=white, draw=black, shape=isosceles triangle, shape border rotate=-90]
\tikzstyle{box}=[fill=white, draw=black, shape=rectangle, minimum width=0.5cm, minimum height=0.3cm]
\tikzstyle{widebox}=[fill=white, draw=black, shape=rectangle, minimum width=1.0cm, minimum height=0.3cm]
\tikzstyle{vwidebox}=[fill=white, draw=black, shape=rectangle, minimum width=1.5cm, minimum height=0.3cm]
\tikzstyle{wdiam}=[fill=white, draw=black, shape=diamond]
\tikzstyle{oc}=[fill={rgb,255: red,255; green,128; blue,0}, draw=none, shape=circle, inner sep=0pt, minimum size=4pt]

\tikzstyle{red}=[-, draw=red]
\tikzstyle{blue}=[-, draw=blue]
\tikzstyle{green}=[-, draw=green]
\tikzstyle{green}=[-, draw=green]
\tikzstyle{dash}=[-, dashed]
\tikzstyle{dot}=[-, dotted]
\tikzstyle{none}=[]
\pgfdeclarelayer{nodelayer}
\pgfdeclarelayer{edgelayer}
\pgfsetlayers{edgelayer,main,nodelayer}

\usepackage{array}
\usepackage{enumerate}
\usepackage{xfrac}
\usepackage{mathrsfs}   
\usepackage{wasysym}	

\usepackage{hyperref}
\hypersetup{
    colorlinks=true,
    linkcolor=blue,
    filecolor=purple,      
    urlcolor=[rgb]{0 0 .6},
}

\newcommand{\1}{\mathds{1}}
\newcommand{\s}{\mathcal}
\newcommand{\bb}{\mathbb}
\newcommand{\dash}{\text{-}}

\newcommand{\id}{\textbf{id}}
\newcommand{\im}{\mathrm{im}}
\newcommand{\infl}{\mathrm{infl}}
\newcommand{\Infl}{\mathrm{Infl}}

\newcommand{\Hom}{\mathrm{Hom}}

\newcommand{\Inv}{\mathrm{Inv}}
\newcommand{\Mod}{\text{-}\mathrm{Mod}}
\newcommand{\Bim}{\text{-}\mathrm{Bim}}

\newcommand{\Gal}{\mathrm{Gal}}

\newcommand{\Aut}{\mathrm{Aut}}
\newcommand{\Br}{\mathrm{Br}}

\newcommand{\mFus}{\mathbf{mFus}}

\newcommand{\Pic}{\mathrm{Pic}}
\newcommand{\End}{\mathrm{End}}

\renewcommand{\Vec}{\mathrm{Vec}}

\newcommand{\colim}{\mathop{\mathrm{colim}}}
\renewcommand{\lim}{\mathop{\mathrm{lim}}}

\makeatletter
\g@addto@macro\th@plain{\thm@headpunct{}}
\g@addto@macro\th@definition{\thm@headpunct{}}
\g@addto@macro\th@remark{\thm@headpunct{}}
\makeatother

\makeatletter
\newtheorem*{rep@theorem}{\rep@title}
\newcommand{\newreptheorem}[2]{%
\newenvironment{rep#1}[1]{%
 \def\rep@title{#2 \ref{##1}}%
 \begin{rep@theorem}}%
 {\end{rep@theorem}}}
\makeatother

\newtheorem{theorem}{Theorem}[section]
\newreptheorem{theorem}{Theorem}
\newtheorem{proposition}[theorem]{Proposition}

\newtheorem{corollary}[theorem]{Corollary}
\newtheorem{lemma}[theorem]{Lemma}

\theoremstyle{definition}
\newtheorem{definition}[theorem]{Definition}
\newtheorem{example}[theorem]{Example}

\newtheorem{remark}[theorem]{Remark}

\tikzset{Rightarrow/.style={double equal sign distance,>={Implies},->},
triple/.style={-,preaction={draw,Rightarrow}},
quadruple/.style={preaction={draw,Rightarrow,shorten >=0pt},shorten >=1pt,-,double,double
distance=0.2pt}}

\title{Invertible Fusion Categories}
\author{Sean Sanford and Noah Snyder}
\date{}

\begin{document}

\begin{abstract}

    A tensor category $\s C$ over a field $\bb K$ is said to be invertible if there's a tensor category $\s D$ such that $\s C \boxtimes \s D$ is Morita equivalent to $\Vec_{\bb K}$.
    When $\bb K$ is algebraically closed, it is well-known that the only invertible fusion category is $\Vec_{\bb K}$, and any invertible multi-fusion category is Morita equivalent to $\Vec_{\bb K}$. By contrast, we show that for general $\bb K$ the invertible multi-fusion categories over a field $\bb K$ are classified (up to Morita equivalence) by $H^3(\mathbb K;\bb G_m)$, the third Galois cohomology of the absolute Galois group of $\bb K$.  We explicitly construct a representative of each class that is fusion (but not split fusion) in the sense that the unit object is simple (but not split simple). One consequence of our results is that fusion categories with braided equivalent Drinfeld centers need not be Morita equivalent when this cohomology group is nontrivial.
\end{abstract}

\maketitle
\tableofcontents

\section{Introduction}

One of the central notions in algebra is the Brauer group of a field $\mathbb{K}$, which is the group (under tensor product) of invertible $\mathbb{K}$-algebras up to Morita equivalence. Here an algebra $A$ is invertible if there's another algebra $B$ and a Morita equivalence $A \otimes B \sim \mathbb{K}$. Multi-fusion categories can be thought of as a categorical analogue of associative algebras, and the notions of Morita equivalence and Brauer group have natural categorical analogues. In particular, a multi-fusion category $\mathcal{C}$ is invertible if there's another multi-fusion category $\mathcal{D}$ such that $\mathcal{C} \boxtimes \mathcal{D}$ is categorically Morita equivalent to $\Vec_{\mathbb{K}}$. Again Morita equivalence classes of invertible multi-fusion categories form a group which can be thought of as a higher categorical analogue of the Brauer group which we denote $\Inv(\bb K)$.

If $\bb K$ is algebraically closed, then $\Inv(\bb K)$ is the trivial group. That is, every invertible multi-fusion category is Morita equivalent $\Vec_{\bb K}$. The goal of this paper is to completely understand $\Inv(\bb K)$ in the case where the base field $\bb K$ is not algebraically closed. In particular, we compute this group, and we construct an explicit representative of each Morita class. More specifically, one of the fundamental results about the Brauer group is that $\Br(\bb K) \cong H^2(\bb K;\bb G_m)$, and we show that $\Inv(\bb K) \cong H^3(\bb K;\bb G_m)$.

This work fits into two larger research programs. Firstly, this is part of the first author's program to study multi-fusion categories over non-algebraically closed fields
\cite{sanford2024fusion, sanfordThesis, plavnik2023tambarayamagami}. Secondly, this is part of a large program \cite{etingofFusionCategoriesHomotopy2009,grossmanCyclicExtensionsFusion2015,grossmanCyclicExtensionsFusion2015, MR3778972,bonteaPointedBraidedTensor2017, nikshychCategoricalLagrangianGrassmannians2013, davydovWittGroupNondegenerate2011, davydovStructureWittGroup2011, MR4406261, brochierInvertible, johnson-freydClassificationTopologicalOrders2020, reutterMinimalNondegenerateExtensions} to think about tensor categories not individually but rather as forming a ``space'' whose points are tensor categories and whose homotopy groups capture essential concepts in tensor category theory. In particular, $\Inv(\bb K)$ is $\pi_0$ of the space of multi-fusion categories (i.e. the core of the Morita $3$-category of fusion categories). This homotopy theoretic viewpoint has been very successful in classifying extensions and classifying minimal modular extensions. Moreover, this viewpoint opens the door to using algebraic topology to study fusion categories by, for example, looking at the long exact sequences or spectral sequences.

If $\bb L$ is a finite Galois extension of $\bb K$ and $\omega$ in $Z^3(G;\bb L^\times)$ is a Galois $3$-cocycle, we construct a category $\Vec_{\bb L}^\omega\big(\Gal(\bb L/\mathbb K)\big)$ which is a Galois-twisted version of $\Gal(\bb L/\mathbb K)$-graded vector spaces over $\bb L$.

\begin{reptheorem}{Thm:ExamplesAreInvertible}
    For any Galois $3$-cocycle $\omega \in Z^3(G;\bb L^\times)$, the multi-fusion category $\Vec_{\bb L}^\omega\big(\Gal(\bb L/\mathbb K)\big)$ is invertible. This assignment $\psi_{\bb L}^0:\omega \mapsto \Vec_{\bb L}^\omega\big(\Gal(\bb L/\mathbb K)\big)$ induces a group homomorphism on cohomology classes 
    $\psi_{\bb L}:H^3(\Gal(\bb L/\mathbb K);\bb L^\times) \rightarrow \Inv(\bb K)$.
\end{reptheorem}

The group $H^3(\bb K;\bb G_m)$ is the colimit of $H^3\big(\Gal(\bb L/\bb K);\bb L^\times\big)$ under inflation maps (see Section \ref{Sec:GaloisCohomology} for more details). 

The main construction that we use in this paper is something called categorical inflation. Categorical inflation from $\mathbb{E}$ to $\mathbb{L}$ starts with a fusion category $\s C$ over $\mathbb{K}$ where the endomorphism of the unit object is $\mathbb{E}$ and returns a Morita equivalent fusion category $\Infl_{\bb E}^{\bb L}(\s C)$ where the endomorphism ring of the unit object is $\mathbb{L}$. This construction categorifies classical inflation in the following sense.

\begin{reptheorem}{Thm:CategorifiesInflation}
    Suppose that $\bb{L}\supseteq\bb{E}\supseteq\bb{K}$ is a tower of Galois extensions, then 
    \[ \Infl^{\bb L}_{\bb E}\Big(\Vec_{\mathbb E}^\omega\big(\Gal(\mathbb E/\mathbb K)\big)\Big) \mathop{\simeq}\limits^{\otimes} \Vec_{\mathbb L}^{\infl_{\bb E}^{\bb L}(\omega)}\left(\Gal(\mathbb L/\mathbb K)\right)\,,\]
    where $\infl_{\bb E}^{\bb L}(\omega)$ is cohomological inflation as in Definition \ref{Def:ClassicalInflation}.
\end{reptheorem}

In particular, replacing a Galois cohomology class by its inflation does not change the Morita equivalence class of the corresponding category.

\begin{reptheorem}{Thm:H3Classification}
    The above maps $H^3(\Gal(\bb L/\mathbb K);\bb L^\times) \rightarrow \Inv(\bb K)$ assemble into an isomorphism
    \[\Psi:H^3(\bb K;\bb G_m) \mathop{\longrightarrow}\limits^{\cong} \Inv(\bb K) \]
    between the third absolute Galois cohomology group of $\mathbb K$ and the group of invertible multi-fusion categories over $\mathbb K$ up to Morita equivalence.
\end{reptheorem}

In particular, two categories $\Vec_{\bb L}^\omega\big(\Gal(\bb L/\mathbb K)\big)$ and $\Vec_{\bb L'}^{\omega'}\big(\Gal(\bb L'/\mathbb K)\big)$ are Morita equivalent if and only if $\infl(\omega)\cdot\infl(\omega')^{-1}$ is trivial in $H^3(\Gal(\mathbb F/\mathbb K);\mathbb F^\times)$ for some extension $\bb F$ of $\bb L$ and $\bb L'$.

Note that the higher homotopy groups of the space of invertible multi-fusion categories are already known, $\pi_1$ is the Brauer group of $\bb K$, $\pi_2$ is trivial, and $\pi_3 \cong \mathbb{K}^\times$ (see e.g. \cite{etingofFusionCategoriesHomotopy2009}. In particular, using the cohomological description of the Brauer group and Hilbert's theorem 90, we have that $\pi_{3-n} \cong H^n(\bb K;\bb G_m)$ for $0 \leq n \leq 3$. It is possible to compute all these homotopy groups abstractly by using an Atiyah–Hirzebruch spectral sequence as in \cite[\S V.C]{johnson-freydClassificationTopologicalOrders2020}. In particular, this spectral sequence gives an abstract isomorphism between $H^3(\bb K;\bb G_m)$ and $\Inv(\bb K)$ as in our main theorem. But the approach we take in this paper is constructive and direct.

We also give an application of this result. It's straightforward to see that if two fusion categories are Morita equivalent categories, then their Drinfeld centers are braided equivalent.
A result of Etingof-Nikshych-Ostrik \cite[Thm 3.1]{etingofWeaklyGroup} says that the converse is true provided you're working over an algebraically closed field.  Fields $\bb K$ for which $H^3(\bb K;\bb G_m)\neq0$ exist, and therefore by Theorem \ref{Thm:H3Classification}, these fields admit fusion categories $\s C$ such that $\s Z(\s C)=\Vec_{\bb K}$, and yet $\s C$ fails to be Morita equivalent to $\Vec_{\bb K}$.  In this way, our results show that this converse breaks down over arbitrary fields, and the group $H^3(\bb K;\bb G_m)$ is precisely the obstruction.

In Section \ref{Sec:Preliminaries} we recall some background materials on division algebras, Galois cohomology, fusion categories over general fields, Morita equivalence, invertibility, and graded fusion categories. In Section \ref{sec:InvertibleMor} we introduce our main technique, categorical inflation, which we use to prove Theorem \ref{Thm:ExamplesAreInvertible} and to prove that the map out of the colimit in Theorem \ref{Thm:H3Classification} is well-defined. In Section \ref{Sec:Splitting} we prove that the map in Theorem \ref{Thm:H3Classification} is surjective, by showing that if you categorically inflate a fusion category to a field that splits all the endomorphism algebras then the result is a category of Galois twisted graded vector spaces. In Section \ref{sec:Obstruction} we prove that the map in Theorem \ref{Thm:H3Classification} is injective by analyzing some old ideas of Eilenberg-Mac Lane and Teichm\"uller, then reinterpreting them within the Morita theory of fusion categories.  Finally, in Section \ref{Sec:spectral} we discuss the connections between our result and the Atiyah-Hirzebruch spectral sequence.

\subsection{Acknowledgements}
We would like to thank Dmitri Nikshych for a conversation at Banff that eventually lead to this project.
We would like to thank David Reutter who pointed out the spectral sequence discussed in Section \ref{Sec:spectral}, Theo Johnson-Freyd for explaining his work on a related computation, and Thibault D\'ecoppet for additional insight into the differentials. This paper is part of a research program on fusion categories over arbitrary fields, and we would like to thank Julia Plavnik for suggesting this line of inquiry.  We would also like to thank Victor Ostrik and Alena Pirutka for pointing out the references in Example \ref{Eg:FieldWithNonzeroH3}. This material is based upon work supported by the National Science Foundation under Grant No. DMS-2000093.

\newpage

\section{Preliminaries} \label{Sec:Preliminaries}

\subsection{Galois cohomology} \label{Sec:GaloisCohomology}

\begin{definition}
    A topological group $G$ is said to be profinite if it is the limit of an inverse system of discrete finite groups. 
\end{definition}

\begin{remark}
    A profinite group can be endowed with a canonical topology induced by the discrete topology on each finite group.
    A topological group is profinite if and only if it is compact and totally disconnected.
\end{remark}

Consider the Galois cohomology groups $H^n\big(\Gal(\bb L/\bb K)\,;\,\bb L^\times\big)$ as $\bb L$ ranges over all finite Galois extensions of $\bb K$.  This collection of finite Galois extensions forms an inverse system of (discrete) finite groups, and its limit is the profinite group $G_{\bb K}:=\Gal(\bb K^\text{sep}/\bb K)$, also known as the absolute Galois group of $\bb K$.

\begin{definition}
    There is a category $\text{TwGrpMod}$ whose objects are pairs $(G,A)$ of a group $G$ and a $G$-module $A$.  A morphism $(G,A)\to(H,B)$ is a pair of maps $f:H\to G$ and $\varphi:A\to B$ so that
    \[\varphi\big(f(h)\cdot a\big)\;=\;h\cdot\varphi(a)\,.\]
    Composition is $(f,\varphi)\circ(g,\psi)=(gf,\varphi\psi)$.
\end{definition}

Given an abelian algebraic group $\mathbb G$ over $\bb K$, the collection 
$$\big\{\big(\Gal(\bb L/\bb K),\bb G(\bb L)\big)\big\}_{\bb L/\bb K\text{ Galois}}$$ forms a directed system in $\text{TwGrpMod}$, and its colimit is $\big(G_{\bb K};\bb G(\bb K^\text{sep})\big)$.

\begin{proposition}[{cf. \cite[ChI-\S2.2-Prop8, ChII-\S1]{serreGaloisCohomology}}]\label{Prop:AbsoluteCohomologyIsAColimit}
    The continuous cohomology groups of the profinite group $G_{\bb K}=\Gal(\bb K^\text{sep}/\bb K)$ with coefficients in $\bb G(\bb K^\text{sep})$ can be computed by
    \[H^n_\text{cts}\big(G_{\bb K};\bb G(\bb K^\text{sep})\big)\cong \colim H^n\Big(\Gal(\bb L/\bb K);\bb G(\bb L)\Big)\,.\]
\end{proposition}

We are primarily concerned with the case where $\bb G=\bb G_m$ the multiplicative group of units, \emph{i.e.} $\bb G_m(\bb L)=\bb L^\times$.

\begin{definition}\label{Def:ClassicalInflation}
    Given a tower of Galois extensions $\bb L\supseteq \bb E\supseteq K$, there is a corresponding morphism $\big(\Gal(\bb E/\bb K),\bb E^\times\big)\to\big(\Gal(\bb L/\bb K),\bb L^\times\big)$ in $\text{TwGrpMod}$, and the image of this morphism under the functor $H^*$ is called inflation:
    \[\;\infl_{\bb E}^{\bb L}\;:=\;H^*\Big(\big(\Gal(\bb E/\bb K),\bb E^\times\big)\to\big(\Gal(\bb L/\bb K),\bb L^\times\big)\Big)\,.\]
    We simply write $\infl$ when the two field extensions can be inferred from context.
\end{definition}

\begin{definition}
    The absolute Galois cohomology of $\bb K$, with multiplicative coefficients is
    \[H^n(\bb K;\bb G_m)\;:=\;H^n_\text{cts}\big(G_{\bb K};(\bb K^\text{sep})^\times\big)\,.\]
\end{definition}

\begin{example}\label{Eg:BrauerGroup}
    The classical Brauer group $\Br(\bb K)$ of $\bb K$ is the group of (finite dimensional) central simple algebras over $\bb K$, up to Morita equivalence.  Extension of coefficients from $\bb K$ to $\bb L$ gives rise to a homomorphism $\Br(\bb K)\to\Br(\bb L)$.  The kernel of this homomorphism is known as the relative Brauer group $\Br(\bb L/\bb K)$. 
    
    According to Jacobson (see \cite[Intro, p21]{MR3114982}), it was Emmy Noether who first proved that $\Br(\bb L/\bb K)\cong H^2\big(\Gal(\bb L/\bb K);\bb L^\times\big)$, despite the fact that the language of group cohomology was not available to her at the time.  Noether's proof was phrased in terms of factor sets --- now recognizable as Galois cocycles --- and their corresponding crossed-product algebras.  Since all central simple algebras split over $\bb K^\text{sep}$, Noether's result implies that, in modern notation, $\Br(\bb K)\cong H^2(\bb K;\bb G_m)\,.$
\end{example}

The categories we introduce in Section \ref{Sec:GalGradedVS} are essentially a categorification of Noether's crossed-product algebras, and their Morita equivalence classes provide a higher categorical analogue of the Brauer group.

\begin{example}\label{Eg:FieldWithNonzeroH3}
    Let $\mathds k$ be a field of characteristic $0$ and let $\bb K=\mathds k(a,b,c)$, where $a$, $b$, and $c$ are indeterminates over $\mathds k$.  There is a version of the Hochschild-Serre spectral sequence that can be used to analyze the cohomology groups of a variety $V$ over $\bb K$:
    \[E_{2}^{p,q}=H^p\big(\bb K\,;\,H^q(\overline{V};\bb G_m)\big)\Rightarrow H^{p+q}(V;\bb G_m)\,,\]
    where $\overline{V}$ denotes the $\overline{\bb K}=\bb K^{sep}$ points of $V$.

    This gives rise to an exact sequence
    \[\Br(\bb K)\to\Br_1(V)\to H^1(\bb K;\Pic(\overline{V}))\to H^3(\bb K;\bb G_m)\,.\]
    
    Uematsu has shown in \cite[Thm 5.1]{uematsuOnTheBrauerGroup} that whenever $\mathds k$ has a primitive third root of unity, and $V$ is a diagonal cubic, it follows that the first map is surjective, and $H^1(\bb K;\Pic(\overline{V}))\cong\bb Z/3\bb Z$.  Thus the field $\bb K=\mathds k(a,b,c)$ has nonzero $H^3(\bb K;\bb G_m)$.

    Similarly, in \cite{rimanVanishing}, Riman considered the case where $\mathds k=\bb Q^{cyc}$ (the cyclotomic closure of $\bb Q$) and where $V$ is a certain del Pezzo surface of degree 4 (see \emph{loc. cit.}).  The result \cite[Thm 3.1]{rimanVanishing} establishes that again the first map is surjective, and $H^1(\bb K;\Pic(\overline{V}))\cong\bb Z/2\bb Z$.  By the same reasoning as above, it follows that $\bb K=\mathds k(a,b,c)$ has nonzero $H^3(\bb K;\bb G_m)$.
\end{example}

The above examples show that $H^3(\bb K;\bb G_m)$ can be nonzero, but for many fields this group is known to be trivial.

\begin{example}\label{Eg:H3IsZero}
    The group $H^3(\bb K;\bb G_m)$ is trivial whenever $\bb K$ is local, or global, or algebraically closed.  The algebraically closed case is obvious, because the absolute Galois group is trivial.
    
    For any non-Archimedean local field $\bb K$, \cite[Cor 7.2.2]{MR2392026} implies that $H^3(\bb K;\bb G_m)=0$.  The Archimedean local field $\bb R$ has absolute Galois group $\bb Z/2\bb Z$, so the cohomology can be computed to be trivial directly from the periodic projective resolution for cyclic groups.

    The global case can be deduced from considering idèles as in \cite[Case $r=3$, p199]{MR220697}.
\end{example}

\subsection{Fusion categories over general fields}

Since fusion categories are often considered only over algebraically closed fields, it is necessary to give a definition of what a fusion category over $\bb K$ means when $\bb K$ is arbitrary.

In order for the collection of multi-fusion categories to be well-behaved, we would like it to be closed under the operation of taking Deligne tensor products.  When the base field is assumed to be neither algebraically closed nor characteristic zero, semisimplicity is not enough.  

Here we alter the definition of fusion to require a notion called local separability, define Deligne tensor products, and explain why we believe this alteration is necessary.

\begin{definition}\label{Def:LocallySeparable}
    A finite semisimple $\mathbb K$-linear category is said to be locally separable (over $\mathbb K$) if the endomorphism algebra of each simple object is a separable $\mathbb K$ algebra.
\end{definition}

\begin{example}\label{Eg:SemisimpleButNotLocSep}
    Let $\mathbb L/\mathbb K$ be an inseparable field extension of finite degree.  The category $\Vec_{\mathbb L}$ locally separable over $\mathbb L$, but not locally separable over $\mathbb K$.
\end{example}

Locally separable categories admit an alternative characterization that may prove to be more satisfying to certain readers.

\begin{proposition}\label{Prop:ModOfASeparableAlgebra}
    A finite semisimple, $\bb K$-linear category $\s C$ is locally separable if and only if there exists a finite dimensional separable $\bb K$-algebra $A$ so that $\s C\simeq A\Mod$.
\end{proposition}

Since local separability is meant to be a strengthening of semisimplicity, the characterization in Proposition \ref{Prop:ModOfASeparableAlgebra} is the reason why Definition \ref{Def:LocallySeparable} presupposes semisimplicity.

\pagebreak

\begin{definition}\label{Def:multi-fusion}
    A multi-fusion category $\s C$ over a field $\bb K$ is a finite, (semisimple, ) locally separable over $\mathbb K$, rigid monoidal category, whose product $\otimes$ is $\mathbb K$-linear in both variables.  If the unit object $\1$ happens to be simple (resp. split simple), then $\s C$ is said to be fusion (resp. split fusion).
\end{definition}

It should be noted that simplicity of an object is meant in the sense of Schur: an object is simple if it has no nontrivial subobjects.  This allows for simple objects whose endomorphism rings are nontrivial division algebras over the base field.  An object with $\End(X)\cong\bb K$ is said to be \emph{split} or \emph{split-simple}.  A category is said to be \emph{split} if all simple objects are split, and this paper pays particular attention to non-split categories. 
Over an algebraically closed field the notions of fusion and split fusion coincide, and one may need to take care over a non-algebraically closed field to determine which notion is more appropriate.

\begin{example}\label{Eg:LBimK}
A key example of a non-split multi-fusion category over a field $\bb K$ is the the category of $\mathbb L$-bimodules that are merely $\mathbb K$-linear, which we denote $\mathbb L\Bim_{\mathbb K}$. If the extension $\bb L/ \bb K$ is Galois then $\mathbb L\Bim_{\mathbb K}$ is semisimple with simple objects $\bb L_g$ for every Galois automorphism $g\in\Gal(\bb L/\bb K)$. The simple object $\bb L_g$ is isomorphic to $\bb L$ as a left $\bb L$-module, but the right action is twisted by $g$. In particular, the endomorphism algebra of each simple object is $\bb L$ and thus these categories are fusion but not split fusion (over $\mathbb K$). The fact that left and right $\mathbb L$-action do not agree means that $\mathbb L\Bim_{\mathbb K}$ is merely fusion over $\mathbb K$ and not fusion over $\mathbb L$.  
\end{example}

Note that in any fusion category $\s C$ over $\bb K$, $\End(\1)$ is a finite dimensional division algebra over $\bb K$ by Schur's lemma, and by the Eckmann-Hilton argument it is automatically commutative.  Thus $\End(\1)$ is a finite degree field extension of $\bb K$.  More generally, $\End(\1)$ is a product of field extensions whenever $\s C$ is multi-fusion.  The following notation is common in the higher fusion category literature.

\begin{definition}\label{Def:LoopsC}
    The endomorphism algebra of a multi-fusion category $\s C$ is $$\Omega\s C:=\End(\1)\,.$$
\end{definition}

A key feature of the Example \ref{Eg:LBimK} is that the elements of $\Omega\s C$ behave differently on the left and right.  The following definition makes this phenomenon precise.

\pagebreak

\begin{definition}[{\cite[Def 3.6]{plavnik2023tambarayamagami}}]\label{Def:Left&RightEmbeddings}
    Let $\s C$ be fusion over $\bb K$, $e\in\Omega\s C$ and $X$ be any object in $\s C$.  The endomorphisms $e\cdot\id_X,\id_X\cdot e:X\to X$ are defined as the compositions below
    \[\begin{tikzcd}[ampersand replacement=\&]
    	\& X \& X \\
    	{\1\otimes X} \&\&\& X\otimes\1 \\
    	{\1\otimes X} \&\&\& X\otimes\1 \& {.} \\
    	\& X \& X
    	\arrow["{e\cdot\id_X}"', from=1-2, to=4-2]
    	\arrow["{\id_X\cdot e}", from=1-3, to=4-3]
    	\arrow["{\ell_X^{-1}}"', from=1-2, to=2-1]
    	\arrow["{e\otimes\id_X}"', from=2-1, to=3-1]
    	\arrow["{\ell_X}"', from=3-1, to=4-2]
    	\arrow["{r_X^{-1}}", from=1-3, to=2-4]
    	\arrow["{\id_X\otimes e}", from=2-4, to=3-4]
    	\arrow["{r_X}", from=3-4, to=4-3]
    \end{tikzcd}\]
    These define algebra embeddings
    \[(-)\cdot\id_X\,,\,\id_X\cdot(-):\Omega\s C\hookrightarrow\End(X),\]
    that are called the left and right embeddings for $X$.  The naturality of the unitors $\ell$ and $r$ imply that the left and right embeddings factor through the inclusion of the center of $\End(X)$.
\end{definition}

One important consequence of allowing non-split simple objects is the fact that the Deligne tensor product $\boxtimes=\boxtimes_{\bb K}$ uses the operation of Karoubi completion in a nontrivial way.

\begin{definition}
    Given two finite $\bb K$-linear abelian categories $\s C$ and $\s D$, their na\"ive tensor product $\s C\otimes\s D$ is described as follows
    \begin{itemize}
        \item Objects in $\s C\otimes\s D$ are formal symbols $X\boxtimes Y$ consisting of an object $X\in\s C$ and an object $Y\in \s D$.
        \item Morphisms are given by the formula $\Hom_{\s C\otimes\s D}(X_1\boxtimes Y_1,X_2\boxtimes Y_2):=\Hom_{\s C}(X_1,X_2)\otimes_{\bb K}\Hom_{\s D}(Y_1,Y_2)$.
        \item Composition is given in terms of simple tensors by the formula $(f_1\otimes g_1)\circ(f_2\otimes g_2):=(f_1\circ f_2)\otimes(g_1\circ g_2)$.
    \end{itemize}
\end{definition}

This na\"ive tensor product category $\s C\otimes\s D$ is typically not abelian, because certain morphisms can fail to have images.  To correct this issue, a formal completion is necessary.

\begin{definition}[{\cite[cf. \S2.4]{lopezFrancoTensorProducts}}]
    The Deligne tensor product $\s C\boxtimes\s D$ of two finite $\bb K$-linear abelian categories $\s C$ and $\s D$ is defined to be the completion of $\s C\otimes\s D$ under all finite colimits.
\end{definition}

\begin{remark}
    In the locally separable setting, the operation of completion under finite colimits amounts to creating direct sums and images for idempotents.  This operation is often referred to as Cauchy completion, while the term idempotent/Karoubi completion refers to the process of adding images for idempotents only.

    When the categories are not locally separable, it is possible that not all images of morphisms appear as images of idempotents in the na\"ive category.  Specifically the images of nilpotent endomorphisms can have this issue.  This is when the full power of completion under \emph{all} finite colimits is necessary in order to make the resulting category abelian.
\end{remark}

\subsection{Morita equivalence}

We assume that the reader is familiar with module categories and bimodule categories following \cite[Ch 7]{etingof2015tensor}.

\begin{definition}
    Let $\mathcal C$ and $\mathcal D$ be multi-fusion categories over $\mathbb K$, and let $\mathcal M$ be a $\mathcal C\dash\mathcal D$-bimodule category. Then $\mathcal M$ is said to be a Morita equivalence (also called an invertible bimodule) if there exists a $\mathcal D\dash\mathcal C$-bimodule category $\mathcal{N}$ and bimodule equivalences $\mathcal{M} \boxtimes_{\mathcal{D}} \mathcal{N} \cong \mathcal{C}$ and $\mathcal{N} \boxtimes_{\mathcal{C}} \mathcal{M} \cong \mathcal{D}$. Here $\mathcal{N}$ is called the inverse bimodule to $\mathcal{M}$, and vice versa.

    When a Morita equivalence between $\s C$ and $\s D$ exists, we say that $\s C$ and $\s D$ are Morita equivalent and write $\s C\sim\s D$.
\end{definition}

\begin{definition}
    Let $A$ be a (unital) algebra in a multi-fusion category $\mathcal C$ over $\mathbb K$.  The symbols $_A\mathcal C$, $\mathcal C_A$, and $_A\mathcal C_A$ will denote the categories of left modules, right modules, and bimodules for $A$ in $\mathcal C$ respectively.
\end{definition}

Note that with this notation, the categories $\mathbb L\Bim_{\mathbb K}$ from Example \ref{Eg:LBimK} can be written as $_{\mathbb L}(\Vec_{\mathbb K})_{\mathbb L}$.

\begin{example}
    Given an algebra $A$ in a multi-fusion category $\mathcal C$, the category $\mathcal C_A$ is a left $\mathcal C$-module category.  If $\s C$ is fusion, then $\s C_A$ is faithful as a $\s C$-module category.  If $A$ is such that $_A\mathcal C_A$ is multi-fusion, and $\s C_A$ is a faithful $\s C$-module category, then $\mathcal C_A$ is a Morita equivalence between $\mathcal C$ and $_A\mathcal C_A$.
\end{example}

\begin{example}\label{Eg:OstriksThm&MoritaEq}
    Given a multi-fusion category $\s C$ over $\bb K$ and a semisimple module category $\s M$, Ostrik's theorem \cite[Thm 1]{ostrikModuleCatsWeakHopf} says that there is an algebra $A$ in $\s C$, and an equivalence of left module categories $\s M\to\s C_A$.  Moreover, if $\s M$ is an invertible $\s C\dash\s D$-bimodule, then $\s D$ is equivalent to ${}_{A}\s C_A$
\end{example}

We have the following definition from \cite[\S 2.4]{douglasDualizable}

\begin{definition}
    Suppose that $\mathcal{M}$ is a $\mathcal C\dash\mathcal D$-bimodule. Let $\mathcal{M}^*$ be $M^{\mathrm{op}}$, and we use the notation $m^*$ to refer to the object in $\mathcal{M}^*$ corresponding to $m$. We can endow $\mathcal{M}^*$ with the structure of a $\mathcal D\dash\mathcal C$-bimodule with the action given by 
    $$d \triangleright m^* \triangleleft c := \left({}^*c \triangleright m \triangleleft {}^*d\right)^*.$$
    We define ${}^*\mathcal{M}$ similarly.
\end{definition}

In \cite[\S 2.4]{douglasDualizable} it is shown that $\mathcal{M}^*$ is the right dual to $\mathcal{M}^*$ and ${}^*\mathcal{M}$ is the left dual to $\mathcal{M}$. This implies the following fact (see also \cite[\S 3.2]{etingofFusionCategoriesHomotopy2009}).

\begin{lemma}
    If $\mathcal M$ is an invertible bimodule, then the inverse bimodule is bimodule equivalent to $\mathcal M^*$ (and also equivalent to ${}^*\mathcal M$).
\end{lemma}

\begin{proposition}\cite[Prop 4.2]{etingofFusionCategoriesHomotopy2009}\label{Prop:MoritaTFAE(Fusion)}
    When $\s C$ and $\s D$ are fusion over $\bb K$, and $\s M$ is a $\s C\dash\s D$-bimodule, the following are equivalent.
    \begin{enumerate}[(i)]
        \item $\s M$ is invertible,
        \item There is a $\s D$-bimodule equivalence $\s M^*\boxtimes_{\s C}\s M\simeq\s D$
        \item There is a $\s C$-bimodule equivalence $\s M\boxtimes_{\s D}\s M^*\simeq\s C$
        \item The action functor $R:\mathcal{D}^\text{mp} \rightarrow \mathrm{End}_{\mathcal{C}}(\mathcal{M})$ given by $R(d)=(-)\triangleleft d$ is a monoidal equivalence.
        \item The action functor $L:\mathcal{C} \rightarrow \mathrm{End}(\mathcal{M})_{\mathcal{D}}$ given by $L(c)=c\triangleright(-)$ is a monoidal equivalence. 
    \end{enumerate}
\end{proposition}

\begin{proof}
    Suppose $\s C$ and $\s D$ are fusion and $\s M$ is a Morita equivalence between them.  Since $\s C$ is locally separable, and $\s D$ is semisimple, it follows from \cite[Thm A.23]{sanford2024fusion} that $\s M$ is locally separable.
    
    With all relevant categories being locally separable, the rest of proof \emph{loc. cit.} is valid as written.
\end{proof}

Note that the fusion assumption is essential in Proposition \ref{Prop:MoritaTFAE(Fusion)}, we will also use a slight generalization to the multi-fusion setting.

\begin{lemma} \label{Lem:FaithfulAndCommutant} \
    Suppose that $\s C$ and $\s D$ are multi-fusion and $\mathcal{M}$ is a $\mathcal C\dash\mathcal D$-bimodule, then $\mathcal{M}$ is invertible iff the following two conditions hold
    \begin{enumerate}[(i)]
        \item $\mathcal{M}$ is faithful as a $\mathcal{C}\dash$module
        \item the action functor $R:\mathcal{D}^\text{mp} \rightarrow \mathrm{End}_{\mathcal{C}}(\mathcal{M})$ is an equivalence.
    \end{enumerate}
\end{lemma}

\begin{proof}
    The only difference between this case and the fusion case is the possibility that $\s M$ can fail to be faithful.  When $\s M$ is invertible both of these properties are easily seen to hold.  Conversely, suppose that \emph{(i)} holds.  The proof of \cite[Thm 7.12.11]{etingof2015tensor} requires no alteration to apply in our setting, and so $\s M$ is a Morita equivalence between $\s C$ and $\End_{\s C}(\s M)^{mp}$.  If \emph{(ii)} also holds, then it follows that $\s M$ is a Morita equivalence.
\end{proof}

\begin{definition}
    Given two module categories $\s M$ and $\s N$ over a multi-fusion category $\s C$, we can form the direct sum $\s M\oplus\s N$.  As a category it is equivalent to the direct product $\s M\times\s N$, and it comes equipped with projection and inclusion functors
    \begin{gather*}
        P_{\s M}:\s M\oplus\s N\to\s M\hspace{5mm}I_{\s M}:\s M\to\s M\oplus\s N\\
        P_{\s N}:\s M\oplus\s N\to\s N\hspace{5mm}I_{\s N}:\s N\to\s M\oplus\s N\,,
    \end{gather*}
    that are all $\s C$-module functors, and such that $P_{\s M}I_{\s M}=\id_{\s M}$, $P_{\s N}I_{\s N}=\id_{\s N}$, and $\id_{\s M\oplus\s N}=I_{\s M}P_{\s M}\oplus I_{\s N}P_{\s N}$.
    
    A $\s C$-module category is said to be decomposable if it is equivalent (as a $\s C$-module category) to $\s M\oplus\s N$, where $\s M$ and $\s N$ are both nonzero.  A $\s C$-module category is said to be indecomposable if it is not decomposable.  A $\s C$-$\s D$-bimodule category is said to be indecomposable if it is indecomposable when thought of as a left $\s C\boxtimes\s D^{mp}$-module category.
\end{definition}

\begin{proposition}\label{Prop:TFAEIndecModules}
    Given a multi-fusion category $\s C$ and a semisimple module category $\s M$, the following are equivalent
    \begin{enumerate}[(i)]
        \item $\s M$ is indecomposable,
        \item for every pair of simple objects $M,N\in \s M$, there exists an object $X\in\s C$ so that $M\hookrightarrow X\triangleright N$,
    \end{enumerate}
\end{proposition}

\begin{example}\label{Eg:IndecVecModules}
    Let $\bb L/\bb K$ be any Galois extension.  The only simple object in $\Vec_{\bb L}$ is $\bb L=\1$, and this must act by the identity functor on any module category.  Thus a module category for $\Vec_{\bb L}$ is indecomposable if and only if it is rank 1, \emph{i.e.} has a unique simple object.
\end{example}

\begin{definition}
    A multi-fusion category over $\bb K$ is (in)decomposable if it is (in)decomposable as a bimodule category over itself.
\end{definition}

\begin{example}
    Let $\s C$ be a fusion category over $\bb K$.  The matrix category $M_n(\s C)$ is $\End(\s C_{\s C}^{\oplus n})$ of right $\s C$-module endofunctors of $\s C^{\oplus n}$.  Composition of functors endows $M_n(\s C)$ with a monoidal structure.  If we let $\1_i$ be the copy of $\1\in\s C$ that comes from the $i^\text{th}$ summand in $\s C^{\oplus n}$, then we can define a right $\s C$-module functor $F_{i,j}$ by the rule that $F_{i,j}(\1_k)=\delta_{j,k}\1_i$.  It follows that every simple object in $M_n(\s C)$ is of the form $X\otimes F_{i,j}(-)$ for some simple $X\in\s C$ and some pair $(i,j)$.  Direct computation shows that $(X\otimes F_{i,j})^*=(X^*\otimes F_{j,i})$, so $M_n(\s C)$ is rigid.  Finally $M_n(\s C)$ is locally separable, because $\End(X\otimes F_{i,j})\cong\End(X)$, and $\s C$ was fusion to begin with.  Thus $M_n(\s C)$ is a multi-fusion category over $\bb K$.

    The unit object of $M_n(\s C)$ is $\bigoplus_i^nF_{i,i}$, and so $M_n(\s C)$ is multi-fusion and not fusion whenever $n>1$.  Furthermore, $M_n(\s C)$ is indecomposable.  To see this, let $X\otimes F_{i,j}$ and $Y\otimes F_{k,l}$ be two simple objects in $M_n(\s C)$.  By using the coevaluation for ${}^*Y$, there is an obvious inclusion
    \[X\otimes F_{i,j}\hookrightarrow\big((X\otimes{}^*Y)\otimes F_{i,k}\big)\circ(Y\otimes F_{k,l})\circ(\1\otimes F_{l,j})\,,\]
    and hence $M_n(\s C)$ is indecomposable as a bimodule over itself by Proposition \ref{Prop:TFAEIndecModules} part \textit{(ii)}.
\end{example}

\subsection{Invertible multi-fusion categories}

\begin{definition}\label{Def:Invertibility}
    A multi-fusion category $\s C$ is called invertible if there exists another multi-fusion category $\s D$ and a Morita equivalence between $\s C \boxtimes \s D$ and $\Vec_{\bb K}$.

    The property of being invertible is a Morita invariant, and Morita equivalence classes of invertible multi-fusion categories form a group $\Inv(\bb K)$ under Deligne tensor product.
\end{definition}

This is the usual notion of invertibility in the monoidal $3$-category of multi-fusion categories, separable bimodules, bimodule functors, and bimodule natural transformations (see \cite[Def-Prop 1.1]{brochierDualizabilityBraidedTensor2021} and \cite{brochierInvertible}).

\begin{example}\label{Eg:Closed=>Inv(k)=1}
    When $\bb K$ is separably, or algebraically closed, the only invertible multi-fusion categories are $M_n(\Vec_{\bb K})$, and so the only invertible \emph{fusion} category over $\bb K$ is $\Vec_{\bb K}$.  In other words, $\Inv(\bb K)=1$.
\end{example}

\begin{definition}\label{Def:Mopposite}
    Given a monoidal category $\s C$, the monoidal opposite $\s C^\text{mp}$ has the same underlying category as $\s C$, but has the monoidal product $X\otimes^\text{mp}Y:=Y\otimes X$.  The associator is given by $\alpha^\text{mp}_{X,Y,Z}:=\alpha_{Z,Y,X}^{-1}$.
\end{definition}

\begin{lemma}
    If $\s C$ is invertible, then the inverse is $\s{C}^{mp}$ and the Morita equivalence is $\s C$ with $\s{C} \boxtimes \s{C}^{mp}$ acting by left and right multiplication.
\end{lemma}
\begin{proof}
    This follows immediately from \cite[Prop. 3.12]{douglasDualizable} which says that $\s{C}^{mp}$ is the dual to $\s{C}$ with evaluation given by $\s C$ with $\s{C} \boxtimes \s{C}^{mp}$ acting by left and right multiplication.
\end{proof}

\begin{theorem}\label{Thm:InvertibleIFFZ(C)=Vec} 
    If a (locally-separable) multi-fusion category is invertible then inclusion of the unit $\Vec_{\bb K} \rightarrow \s Z(\s C)$ is an equivalence.
\end{theorem}
\begin{proof}
    By the previous Lemma, we want to check that $\s C$ is a Morita equivalence between $\s C \boxtimes \s{C}^{mp}$ and $\Vec_{\bb K}$. By Lemma \ref{Lem:FaithfulAndCommutant} part (ii), $\s{C}$ being a Morita equivalence implies that the following functor is also an equivalence:
    $$\Vec_{\bb K} \rightarrow \s C \boxtimes_{\s{C} \boxtimes \s{C}^{mp}} \s{C}^{op} = \s Z(\s C)\,.$$
\end{proof}

The converse to Theorem \ref{Thm:InvertibleIFFZ(C)=Vec} is also true, but we will not need it here.  Since it is not necessary for our current computations and requires more sophisticated Galois theory, we have chosen to relegate it to a follow up paper.

\begin{remark}
    There's a more general condition for invertibility in \cite[Thm. 2.27]{brochierInvertible}, which requires a factorizability condition in addition to the triviality of the center. We only require one of these two conditions because of the double-commutant Lemma \ref{Lem:FaithfulAndCommutant}.
\end{remark}

\subsection{Graded fusion categories}\label{Sec:GradedFusion}

A $\bb K$-linear category $\s C$ is said to be graded by a group $G$, if $\s C=\bigoplus_{g\in G}\s C_g$, where the components $\s C_g\subseteq\s C$ are full subcategories.  If $\s C$ is monoidal, then we can also ask that the grading respects the monoidal structure, in the sense that for any $X\in\s C_g$ and $Y\in\s C_h$, $X\otimes Y\in\s C_{gh}$.  A grading is said to be faithful if $\s C_g\neq0$ for every $g\in G$.

\begin{proposition}\label{Prop:GradedComponentsInvertible}
    Let $\s C$ be fusion over $\mathbb K$, and suppose that $\s C=\bigoplus_{g\in G}\s C_g$ is a faithful, monoidal grading of $\s C$.  It follows that every graded component $\,\s C_g$ is an invertible $\s C_1$-bimodule category, \emph{i.e.} a Morita equivalence from $\s C_1$ to itself.
\end{proposition}

The following proof is adapted from \cite[Ch 7]{etingof2015tensor}.

\begin{proof}
    Suppose $Y,Z\in\s C_g$ are simple objects.  By monoidality of the grading, the dual $Y^*$ is automatically an object in $\s C_{g^{-1}}$, and so $(Z\otimes Y^*)\in\s C_1$.  Using evaluation for $Y$, we find that there is a nonzero map $(Z\otimes Y^*)\otimes Y\to Z$, and this shows that $\s C_g$ is indecomposable as a left $\s C_1$-module category by Proposition \ref{Prop:TFAEIndecModules}.  A similar argument shows that $\s C_g$ is indecomposable as a right module as well.
    
    Again using the evaluation map $Y^*\otimes Y\to\1$, it follows that $A=Y\otimes Y^*$ has the structure of an algebra object in $\s C_1$, and the coevaluation $\1\to Y\otimes Y^*$ makes this algebra unital.  By Ostrik's theorem \cite[Thm 1]{ostrikModuleCatsWeakHopf}, $\s C_g\simeq(\s C_1)_A$ as left $\s C_1$ module categories via the functor $Z\mapsto Z\otimes Y^*$.  In order to show that $\s C_g$ is invertible, it is sufficient to show that the right action of $\s C_1$ on $\s C_g\simeq (\s C_1)_A$ induces an equivalence $\s C_1\simeq{}_A(\s C_1)_A$ as fusion categories.

    Consider the object $P:=Y^*\otimes_{A}\otimes Y$.  The evaluation map $Y^*\otimes Y\to\1$ is $A$-balanced, and thus factors through a map $P\to\1$.  This map is surjective, because the evaluation map is nonzero, and $\1$ is simple.  Observe that
    \[P\otimes P\;=\;Y^*\otimes_{A}\otimes Y\otimes Y^*\otimes_{A}\otimes Y\;=\;Y^*\otimes_{A}A\otimes_{A}Y\;\cong\;P\,.\]
    A simple combinatorial argument \cite[Prop 3.4]{sanford2024fusion} now shows that $P\cong\1$.

    Define a functor by the rule
    \[F:X\mapsto Y\otimes X\otimes Y^*\;,\]
    with $F(f)=\id_{Y}\otimes f\otimes\id_{Y^*}$.  There is an obvious monoidal structure given by the composition
    \begin{align*}
        F(W)\otimes F(X)&=Y\otimes W\otimes Y^*\otimes_{A} Y\otimes X\otimes Y^*\\
        &\cong Y\otimes W\otimes X\otimes Y^*\;=\;F(W\otimes X)\,.
    \end{align*}
    
    A slight modification of the above computation shows that $(Z\otimes Y^*)\otimes_AF(X)\cong(Z\otimes X)\otimes Y^*$.  In other words, $F$ is the functor that identifies the right $\s C_1$ action on $\s C_g$ with the right ${}_A(\s C_1)_A$ action on $(\s C_1)_A$.
    
    There is a clear choice for the inverse functor.
    \[F^{-1}(M)\;=\;Y^*\otimes_AM\otimes_AY\,.\]
    The composition $FF^{-1}$ is easily seen to be isomorphic to the identity functor on ${}_A(\s C_1)_A$.  That $F^{-1}F\cong\id_{\s C_1}$ follows from our previous observation that $P\cong\1$.  Thus $F:\s C_1\to{}_A(\s C_1)_A$ is an equivalence, so $\s C_g$ is an invertible $\s C_1$-bimodule category.
\end{proof}

\begin{remark}\label{Rem:FaithfulnessComment}
    Note that the above proof requires $\s C$ to be fusion and not multi-fusion.  If $\s C$ were multi-fusion, then $P\to\1$ will be surjective precisely when $\s M$ is faithful.  If $\s M$ is not faithful, then $P$ will be a proper summand of $\1$.  When this happens, ${}_{A}(\s C_1)_A$ will be equivalent to the proper subcategory of $\s C_1$ that acts faithfully on $\s M$.
\end{remark}

\begin{proposition}\label{Prop:GaloisGrading}
    If $\s C$ is a fusion category over $\bb K$, and $\End(X)\cong\bb L$ for every simple object $X$ in $\s C$, then $\s C$ admits a monoidal grading
    \begin{gather*}
        \s C\;=\;\bigoplus_{g\in\Aut(\bb L/\bb K)}\s C_g\;,\\
        \s C_g\;:=\;\{X\in \s C~|~\id_X\cdot l\,=\,g(l)\cdot\id_X\,,\,\text{ for every }l\in\bb L\}\,.
    \end{gather*}
\end{proposition}

\begin{proof}
    When all simple objects have the same field as their endomorphism algebra, this implies that the left and right embeddings (see Definition \ref{Def:Left&RightEmbeddings}) $(-)\cdot\id_X\,,\,\id_X\cdot(-):\mathbb L\hookrightarrow\End(X)$ are isomorphisms for every simple object $X$ in $\mathcal C$.  To each simple object $X$, we can associate an automorphism $g_X:\mathbb L\to\mathbb L$ determined by the formula
    \begin{equation}
        g_X(l)\cdot\id_X\;=\;\id_X\cdot l\label{Eqn:Right2Left}
    \end{equation}
    Semisimplicity of $\s C$ implies that, for any given $g\in \Aut(\bb L/\bb K)$, the collection of all simple objects $X$ with $g_X=g$ generate under direct sums a full subcategory
    \[\s C_g\;:=\;\{X\in \s C~|~\id_X\cdot l\,=\,g(l)\cdot\id_X\,,\,\text{ for every }l\in\bb L\}\subseteq\s C\,.\]
    By decomposing any given object $X$ into simples $X_i$, and recording the corresponding $g_{X_i}$'s, it is clear that
    \[\s C\;=\;\bigoplus_{g\in\Aut(\bb L/\bb K)}\s C_g\;.\]
    
    Whenever $X$, $Y$, and $Z$ are simple, and there is a nonzero morphism $X\otimes Y\to Z$, the assignment $X\mapsto g_X$ is monoidal in the sense that $g_X\circ g_Y=g_Z$  (cf. \cite[Lem 3.11]{plavnik2023tambarayamagami}).  Thus for any $X\in\s C_g$ and $Y\in \s C_h$, their tensor product $X\otimes Y$ lies in the subcategory $\s C_{gh}$, so the grading is monoidal.
\end{proof}

\begin{example}\label{Eg:GaloisGradingNotAlwaysFaithful}
    The grading in Proposition \ref{Prop:GaloisGrading} is not necessarily faithful.  For example let $\bb L/\bb K$ be Galois.  Any proper subgroup $H\lneq\Gal(\bb L/\bb K)$ gives rise to a proper full subcategory $\s C_H\subsetneq\bb L\Bim_{\bb K}$ generated by only those $\bb L_h$ where $h\in H$.  The category $\s C_H$ is fusion and $\Gal(\bb L/\bb K)$-graded, though not faithfully.
\end{example}

\begin{example}\label{Eg:NotAlwaysGaloisGraded}
    It should be noted that the left and right embeddings are typically not isomorphisms, and that categories are not always Galois-graded.
    
    Consider the category $\mathbb L\Bim_{\bb K}$ when $\mathbb L=\mathbb Q(\sqrt[3]{2})$ and $\mathbb K=\mathbb Q$, which is a non-normal field extension.  This category has two simple objects $\1$ and $X$, where $\End(X)\cong\mathbb Q(\sqrt[3]{2},\zeta_3)$ is the splitting field for $\mathbb L$.  The fusion rules are $X\otimes X=\1\oplus\1\oplus X$, and so the category has a trivial grading group.

    This failure of the existence of the Galois grading comes from the fact that the left and right embeddings fail to have the same image in $\End(X)$ --- a phenomenon that cannot happen when both embeddings are isomorphisms.
\end{example}

\section{Invertible categories up to Morita equivalence} \label{sec:InvertibleMor}

\subsection{Galois-twisted graded vector spaces}\label{Sec:GalGradedVS}

Recall from Example \ref{Eg:LBimK} that if $\bb L/ \bb K$ is Galois and $G=\Gal(\bb L/\bb K)$, then we have a $G$-graded fusion category $\mathbb L\Bim_{\mathbb K}$. This can be thought of as a Galois-twisted version of the ordinary category of $G$-graded vector spaces, where the twisting is visible in the discrepancy between the left and right embeddings (see Definition \ref{Def:Left&RightEmbeddings}). Just as the category of $G$-graded vector spaces can be twisted to change its associator by introducing a $3$-cocycle, we can generalize $\mathbb L\Bim_{\mathbb K}$ by introducing a Galois $3$-cocycle. These will be the main examples in this paper.

\begin{definition}
If $\omega$ is a cocycle in $Z^3\big( \Gal(\mathbb L/\mathbb K), (\bb L)^\times \big)$, then we let $\Vec_{\mathbb L}^\omega\big(\Gal(\mathbb L/\mathbb K)\big)$ denote the category $\bb L\Bim_{\bb K}$ equipped with the monoidal structure determined by the following diagram
\[\begin{tikzcd}[ampersand replacement=\&]
	{(\mathbb L_{a}\otimes\mathbb L_{b})\otimes\mathbb L_{c}} \&\& {\mathbb L_{a}\otimes(\mathbb L_{b}\otimes\mathbb L_{c})} \\
	{\mathbb L_{ab}\otimes\mathbb L_{c}} \&\& {\mathbb L_{a}\otimes\mathbb L_{bc}} \\
	{\mathbb L_{abc}} \&\& {\mathbb L_{abc}\;\;\;\;.}
	\arrow["{\alpha_{\mathbb L_{a},\mathbb L_{b},\mathbb L_{c}}}", from=1-1, to=1-3]
	\arrow[Rightarrow, no head, from=1-1, to=2-1]
	\arrow[Rightarrow, no head, from=2-1, to=3-1]
	\arrow[Rightarrow, no head, from=1-3, to=2-3]
	\arrow[Rightarrow, no head, from=2-3, to=3-3]
	\arrow["{\omega(a,b,c)\cdot\id_{\mathbb L_{abc}}}"', from=3-1, to=3-3]
\end{tikzcd}\]
\end{definition}

\begin{example}
    If $\omega_0$ is cohomologically trivial, then $\Vec_{\mathbb L}^{\omega_0}\big(\Gal(\mathbb L/\mathbb K)\big) \cong \mathbb L\Bim_{\mathbb K}$. In particular,
    $\Vec_{\mathbb L}^{1}\big(\Gal(\mathbb L/\mathbb K)\big)$ is Morita trivial.
\end{example}

\begin{lemma}\label{Lem:WellDefOnCohomology}
    The cocycles $\omega_1$ and $\omega_2$ are cohomologous if and only if $\Vec_{\mathbb L}^{\omega_1}\big(\Gal(\mathbb L/\mathbb K)\big)$ and $\Vec_{\mathbb L}^{\omega_2}\big(\Gal(\mathbb L/\mathbb K)\big)$ are equivalent as monoidal categories.
\end{lemma}
\begin{proof}
    Let $\tau\in C^2\big(\Gal(\bb L/\bb K);\bb L^\times\big)$ be a cochain witnessing the fact that $\omega_1$ and $\omega_2$ are cohomologous, \emph{i.e.} $\delta(\tau)=\omega_1\cdot\omega_2^{-1}$.  The identity functor can be equipped with a tensor structure $J$ by the formula $J_{\bb L_a,\bb L_b}=\tau(a,b)\cdot\id_{\bb L_{ab}}$.  The monoidality coherence condition on the pair $(\id_{\bb L\Bim_{\bb K}},J)$ is precisely the coboundary condition used to define $\tau$.

    Conversely, suppose there is a monoidal equivalence $(F,J)$ from $\Vec_{\mathbb L}^{\omega_1}\big(\Gal(\mathbb L/\mathbb K)\big)$ to $\Vec_{\mathbb L}^{\omega_2}\big(\Gal(\mathbb L/\mathbb K)\big)$.  By naturality of $J$, the action of $F$ on morphisms is completely determined by the action on $\End(\1)$.  Here the functor must act by a Galois transformation, so let us call this $f:\bb L\to\bb L$.  It follows that $g_{FX}=f\circ g_X\circ f^{-1}$, and hence that $F(X)=\bb L_f\otimes(-)\otimes\bb L_{f}^*$.  The conjugation action of a group on itself famously induces the trivial action on cohomology (see \emph{e.g.} \cite[Prop III.8.3]{MR1324339}).  Since the tensorator can only adjust the cocycle by a coboundary, it follows that $\omega_1$ and $\omega_2$ are cohomologous.
\end{proof}

In particular, we often abuse notation and refer to $\Vec_{\mathbb L}^\omega\big(\Gal(\mathbb L/\mathbb K)\big)$ where $\omega$ is an element of cohomology.

\begin{remark}
    Lemma \ref{Lem:WellDefOnCohomology} stands in contrast to the classification of (Galois trivial) $G$-graded vector spaces.  In the classical case, monoidal equivalences can act by arbitrary automorphisms of $G$.  This causes the monoidal equivalence classes of $\Vec_{\bb K}^{\omega}(G)$ to correspond to the orbits of the action of $\mathrm{Out}(G)$ on $H^3(G;\bb K^\times)$.  In our case, Galois nontriviality permits only inner automorphisms, and thus all the relevant orbits are singletons.
\end{remark}

\begin{proposition}\label{Prop:GrpHomomorphism}
    If $\omega_1, \omega_2 \in H^3\big( \Gal(\mathbb L/\mathbb K), (\bb L)^\times \big)$, then there is a Morita equivalence 
    \[\Vec_{\mathbb L}^{\omega_1}\big(\Gal(\mathbb L/\mathbb K)\big) \boxtimes \Vec_{\mathbb L}^{\omega_2}\big(\Gal(\mathbb L/\mathbb K)\big) \sim \Vec_{\mathbb L}^{\omega_1\cdot\omega_2}\big(\Gal(\mathbb L/\mathbb K)\big). \]
\end{proposition}

\begin{proof}
    The monoidal unit in $\s C=\Vec_{\mathbb L}^{\omega_1}\big(\Gal(\mathbb L/\mathbb K)\big) \boxtimes \Vec_{\mathbb L}^{\omega_2}\big(\Gal(\mathbb L/\mathbb K)\big)$ is $\1=\bb L\boxtimes\bb L$.  The endomorphisms of the unit are
    \[\End(\1)\cong\bb L\otimes_{\bb K}\bb L\cong\prod_{g\in\Gal(\bb L/\bb K)}\bb L\,.\]
    Let us denote the summand of $\1$ corresponding to $g\in\Gal(\bb L/\bb K)$ by $\1_g$.  The object $\1_g$ is determined by the following formula
    \begin{equation}
        \id_{\1_g}\otimes(\id_X\boxtimes l\cdot\id_Y)\;=\;\id_{\1_g}\otimes(g(l)\cdot\id_X\boxtimes\id_Y)\,.\label{Eqn:DefOf1g}
    \end{equation}
    It follows also that a similar formula holds on the right
    \begin{equation}
        (\id_X\boxtimes\id_Y\cdot l)\otimes\id_{\1_g}\;=\;(\id_X\cdot g(l)\boxtimes\id_Y)\otimes\id_{\1_g}\,.\label{Eqn:1gOnTheRight}
    \end{equation}
    From these formulas, it follows that for any $f,g,h,k\in\Gal(\bb L/\bb K)$,
    \begin{align*}
        &\id_{\1_f}\otimes\big(\id_{\bb L_g}\boxtimes\id_{\bb L_k}\cdot l\big)\otimes\id_{\1_h}\\
        &=\id_{\1_f}\otimes\big(\id_{\bb L_g}\boxtimes k(l)\cdot\id_{\bb L_k}\big)\otimes\id_{\1_h}\\
        &=\id_{\1_f}\otimes\big((fk)(l)\cdot\id_{\bb L_g}\boxtimes\id_{\bb L_k}\big)\otimes\id_{\1_h}\\
        &=\id_{\1_f}\otimes\big(\id_{\bb L_g}\cdot(g^{-1}fk)(l)\boxtimes\id_{\bb L_k}\big)\otimes\id_{\1_h}\\
        &=\id_{\1_f}\otimes\big(\id_{\bb L_g}\boxtimes\id_{\bb L_k}\cdot(h^{-1}g^{-1}fk)(l)\big)\otimes\id_{\1_h}\,.
    \end{align*}
    If $k\neq f^{-1}gh$, then there is some element $l\in\bb L$ so that $(h^{-1}g^{-1}fk)(l)\neq l$.  Thus we find that $\id_{\1_f}\otimes\big(\id_{\bb L_g}\boxtimes\id_{\bb L_k}\big)\otimes\id_{\1_h}=0$ unless $k=f^{-1}gh$.

    By general facts about Deligne products, simple objects in $\s C$ can be written as $\1_f\otimes(\bb L_g\boxtimes\bb L_k)$ for some $f,g,k\in\Gal(\bb L/\bb K)$.  Let us briefly denote $X_{f,g,h}=\1_f\otimes(\bb L_g\boxtimes\bb L_{f^{-1}gh})$.  The above computation implies that $X_{f,g,h}\cong X_{f,g,h}\otimes\1_h$.  This implies the tensor product rule
    \[X_{f,g,h}\otimes X_{i,j,k}\;=\;\delta_{h=i}\cdot X_{f,gj,k}\,.\]
    This rule shows that $\s C$ is indecomposable, with a diagonal summand $\mathcal C_{1,1}$ generated by the objects of the form $X_{1,g,1}$.  Indecomposable multi-fusion categories are Morita equivalent to any of their diagonal summands, and so it will suffice to show that $\s C_{1,1}\simeq\Vec_{\mathbb L}^{\omega_1\omega_2}\big(\Gal(\mathbb L/\mathbb K)\big)$.

    We can define a functor $\s C_{1,1}\to\Vec_{\mathbb L}^{\omega_1\omega_2}\big(\Gal(\mathbb L/\mathbb K)\big)$ by sending $X_{1,g,1}$ to $\bb L_{g}$, and having all the tensorators be identities.  The fusion rules, endomorphism algebras and left/right embeddings already match and so it remains to show that the associator is correct.  Direct computation shows that 
    \[\begin{tikzcd}[ampersand replacement=\&,column sep=45]
    	{(X_{1,g,1}\otimes X_{1,h,1})\otimes X_{1,k,1}} \&\& {X_{1,g,}\otimes(X_{1,h,1}\otimes X_{1,k,1})} \\
    	{X_{1,gh,1}\otimes X_{1,k,1}} \&\& { X_{1,g,1}\otimes X_{1,hk,1}} \\
    	{X_{1,ghk,1}} \&\& {X_{1,ghk,1}\;\;.}
    	\arrow["{\alpha_{ X_{1,g,1}, X_{1,h,1}, X_{1,k,1}}}", from=1-1, to=1-3]
    	\arrow[Rightarrow, no head, from=1-1, to=2-1]
    	\arrow[Rightarrow, no head, from=2-1, to=3-1]
    	\arrow[Rightarrow, no head, from=1-3, to=2-3]
    	\arrow[Rightarrow, no head, from=2-3, to=3-3]
    	\arrow["{\id_{\1_1}\otimes\big(\omega_1(g,h,k)\cdot\id_{\mathbb L_{ghk}}\boxtimes\omega_2(g,h,k)\cdot\id_{\mathbb L_{ghk}}\big)}"', from=3-1, to=3-3]
    \end{tikzcd}\]
    By applying Equation \ref{Eqn:DefOf1g}, the scalar $\omega_2(g,h,k)$ can be pushed to the left $\boxtimes$-factor to show that the coefficient of the associator is exactly $\omega_1(g,j,k)\omega_2(g,h,k)$, and this completes the proof.
\end{proof}

Our first main theorem now follows as a consequence of Lemma \ref{Lem:WellDefOnCohomology} and Proposition \ref{Prop:GrpHomomorphism}. 

\begin{theorem}\label{Thm:ExamplesAreInvertible}
    For any Galois $3$-cocycle $\omega \in Z^3(G;\bb L^\times)$, the multi-fusion category $\Vec_{\bb L}^\omega\big(\Gal(\bb L/\mathbb K)\big)$ is invertible. The assignment $\psi_{\bb L}^0:\omega \mapsto \Vec_{\bb L}^\omega\big(\Gal(\bb L/\mathbb K)\big)$ induces a group homomorphism on cohomology classes 
    $\psi_{\bb L}:H^3(\Gal(\bb L/\mathbb K);\bb L^\times) \rightarrow \Inv(\bb K)$.
\end{theorem}

\begin{proof}
    From Proposition \ref{Prop:GrpHomomorphism},
    $$\Vec_{\mathbb L}^\omega\big(\Gal(\mathbb L/\mathbb K)\big)\boxtimes\Vec_{\mathbb L}^{\omega^{-1}}\big(\Gal(\mathbb L/\mathbb K)\big)\sim\Vec_{\mathbb L}^1\big(\Gal(\mathbb L/\mathbb K)\big)\,.$$
    The latter category is just $\bb L\Bim_{\bb K}$, which is Morita equivalent to $\Vec_{\bb K}$.  Thus by Definition \ref{Def:Invertibility}, $\psi_{\bb L}^0(\omega)$ is an invertible fusion category, with inverse $\psi_{\bb L}^0(\omega^{-1})$.

    For any 2-cochain $\tau$, lemma \ref{Lem:WellDefOnCohomology} supplies a monoidal equivalence that can be used to construct a Morita equivalence $\psi_{\bb L}^0(\omega)\sim\psi_{\bb L}^0(\omega\cdot\delta(\tau))$.  Thus $[\psi_{\bb L}^0(-)]_{\sim}$ descends to a well-defined map $\psi_{\bb L}$ on cohomology.  Another application of Proposition \ref{Prop:GrpHomomorphism} establishes that this induced map is indeed a group homomorphism.
\end{proof}

\subsection{Categorical inflation}

In this subsection we introduce our key construction, categorical inflation.  Given a fusion category $\s C$ over a field $\mathbb K$ with $\Omega\s C=\mathbb E$ and a finite Galois extension $\mathbb L/\mathbb E$, categorical inflation replaces $\mathcal C$ with a Morita equivalent fusion category $\Infl_{\bb E}^{\bb L}(\mathcal C)$ (still over $\mathbb K$) whose endomorphism field is $\mathbb L$.  

We show that this operation allows us to reduce the Morita classification of invertible fusion categories to invertible \emph{pointed} fusion categories, \emph{i.e.} those for which every simple object is invertible.  This reduction to pointed categories allows for the extraction of cohomological invariants, and we call this categorical inflation because it categorifies inflation in Galois cohomology.

\begin{lemma}\label{Lem:IndecSummands}
    Let $\s C$ be a fusion category over $\bb{K}$ with $\Omega\s C = \bb{E}$, and suppose that $\bb L$ is a separable extension of $\bb{E}$, then $\Vec_\bb{L} \boxtimes \s C$ is a module category over $\s C$.  If $\bb E=\bb K(\theta)$, for some $\theta\in\bb E$ with minimal polynomial $f(X)$, then the indecomposable summands of $\Vec_\bb{L} \boxtimes \s C$ are in bijection with the indecomposable factors of $f(X)$ in $\mathbb L$.
\end{lemma}

\begin{proof}
    The module structure on $\Vec_{\bb L}\boxtimes\s C$ comes from the action of $\mathcal C$ on itself, with trivial action on $\Vec_{\bb L}$.   Suppose $f(X)$ factors as $f(X)=f_1(X)\cdots f_m(X)$ over $\mathbb L$.  We can compute that
    \[\End(\bb L\boxtimes\1)\cong\mathbb L\otimes_{\mathbb K}\mathbb E\cong\prod_{i=1}^m\frac{\mathbb L[X]}{\langle f_i(X)\rangle}\,.\]
    Each of the factors is a finite degree field extension of $\mathbb L$, and has a corresponding projection $p_i\in\End(\bb L\boxtimes\1)$.
    
    Since all relevant field extensions are separable, the categories we are considering are semisimple, and so the Deligne tensor product is modeled by Karoubi completion.  In other words, the projections $p_i$ correspond to the simple summands $\1_i$ of $\bb L\boxtimes\1$.  Simplicity of the summands corresponds to minimality of the projections, and the projections are minimal because the factors are fields.
    
    These projections provide a canonical decomposition of $\Vec_{\bb L}\boxtimes\s C$ into subcategories:
    \[\Vec_{\bb L}\boxtimes\s C\;=\;(\Vec_{\bb L}\boxtimes\s C)\otimes(\bb L\boxtimes\1)\;=\;\bigoplus_{i=1}^m(\Vec_{\bb L}\boxtimes\s C)\otimes \1_i\;.\]
    Note that the action of $\s C$ preserves each of the summands $(\Vec_{\bb L}\boxtimes\s C)\otimes \1_i$, and so this is a decomposition of module categories.
    
    Let $U,V$ be simple objects in $(\Vec_{\bb L}\boxtimes\s C)$.  By semisimplicity, these correspond to minimal projections $q_U\in\End(\bb L\boxtimes X)$ and $q_V\in\End(\bb L\boxtimes Y)$, for some simple objects $X$ and $Y$ in $\s C$.  The object $U$ is in the subcategory $(\Vec_{\bb L}\boxtimes\s C)\otimes \1_i$ if and only if $q_U\otimes p_i\neq0$.  By rigidity of $\s C$, this is equivalent to $p_i\circ(\id_{\bb L}\boxtimes\text{ev}_{X})\circ(\id_{\bb L\boxtimes X^*}\otimes q_U)\neq0$.  Similarly, $V$ is in $(\Vec_{\bb L}\boxtimes\s C)\otimes \1_i$ if and only if $(\id_{\bb L\boxtimes ^*Y}\otimes q_V)\circ(\id_{\bb L}\boxtimes\text{coev}_{^*Y})\circ p_i\neq0$.  Thus if $U$ and $V$ are both in $(\Vec_{\bb L}\boxtimes\s C)\otimes \1_i$, then the morphism
    \[(\id_{\bb L\boxtimes ^*Y}\otimes q_V)\circ(\id_{\bb L}\boxtimes\text{coev}_{^*Y})\circ p_i\circ(\id_{\bb L}\boxtimes\text{ev}_{X})\circ(\id_{\bb L\boxtimes X^*}\otimes q_U)\]
    is a composition of two nonzero morphisms
    \[\big(\bb L\boxtimes (X^*\otimes X)\big)\to \1_i\to \big(\bb L\boxtimes (^*Y\otimes Y)\big)\;.\]
    Since $\1_i$ is simple, this composition must be nonzero.  Using rigidity of $\s C$ again, this composition being nonzero is equivalent to there being a nonzero morphism $U\hookrightarrow(X\otimes ^*Y)\triangleright V$.  The existence of such nonzero morphisms implies that $(\Vec_{\bb L}\boxtimes\s C)\otimes \1_i$ is indecomposable by Proposition \ref{Prop:TFAEIndecModules}.
\end{proof}

This result is the primary tool that we use to investigate Morita equivalence.  We linger here momentarily to record some important corollaries.  

\begin{corollary}\label{Cor:CanonicalSummand}
    In the context of Lemma \ref{Lem:IndecSummands}, there is a canonical summand $\s M(\bb L/\bb E):=(\Vec_\bb{L} \boxtimes \s C)\otimes \1_1$ corresponding to the inclusion $\bb{E} \subseteq \bb{L}$, and $\End(\1_1)\cong\bb L$.
\end{corollary}

\begin{proof}
    Since $\theta\in\bb E\subseteq\bb L$, one of the factors of $f(X)$ over $\bb L$ is the linear term $(X-\theta)$, and $\bb L[X]/(X-\theta)\cong\bb L$.  Without loss of generality, we can assume that this factor corresponds to the first summand $\1_1$ of $\bb L\boxtimes\1$ as described in Lemma \ref{Lem:IndecSummands}.
\end{proof}

\begin{corollary}\label{Cor:NiceCaseIndecSummands}
    In the context of Lemma \ref{Lem:IndecSummands}, if $\bb L$ contains the normal closure of $\bb E/\bb K$, then summands of $\Vec_\bb{L} \boxtimes \s C$ are in bijection with embeddings $\bb E \hookrightarrow \bb L$.
\end{corollary}

\begin{proof}
    By assumption, $f(X)$ factors as $f(X)=\prod_i^m(X-\theta_i)$ over $\mathbb L$, where the $\theta_i$ are the Galois conjugates of $\theta=\theta_1$ in $\mathbb L$.  We can now compute that
    \[\End(\bb L\boxtimes\1)\cong\mathbb L\otimes_{\mathbb K}\mathbb E\cong\prod_{i=1}^m\frac{\mathbb L[X]}{\langle (X-\theta_i)\rangle}\cong\prod_{i=1}^m\mathbb L\,.\qedhere\]
\end{proof}

\begin{lemma}\label{Lem:MoritaEquivalence}
    The category $\s M(\bb L/\bb E)$ is a Morita equivalence between $\s C$ and $\1_1\otimes(\bb L\Bim_{\bb K}\boxtimes\s C)\otimes\1_1$
\end{lemma}

\begin{proof}
    For any finite degree field extension $\mathbb L/\mathbb K$, the category $\Vec_{\mathbb L}$ is a Morita equivalence between $\Vec_{\mathbb K}$ and $\mathbb L\Bim_{\mathbb K}$.  It follows that if $\mathcal C$ is any fusion category over $\mathbb K$, then $\Vec_{\mathbb L}\boxtimes\mathcal C$ is a Morita equivalence between $\mathcal C$ and $\mathbb L\Bim_{\bb K}\boxtimes\mathcal C$.
    
    In general, $\bb L\Bim_{\bb K}\boxtimes \s C$ is multi-fusion, with $\bb L\boxtimes\1$ decomposing as in Lemma \ref{Lem:IndecSummands}.  The object $\1_1$ is an algebra in $\bb L\Bim_{\bb K}\boxtimes \s C$, and the corresponding category of right modules is $(\bb L\Bim_{\bb K}\boxtimes \s C)\otimes\1_1$.  The module dual of $\bb L\Bim_{\bb K}\boxtimes \s C$ with respect to this module category is the category of $\1_1$ bimodules, which is equivalent to the diagonal subcategory $\1_1\otimes(\bb L\Bim_{\bb K}\boxtimes \s C)\otimes\1_1$.  The category $\s M(\bb L/\bb E)$ is simply the composition of the two invertible bimodules $\Vec_{\bb L}\boxtimes\s C$, and $(\bb L\Bim_{\bb K}\boxtimes \s C)\otimes\1_1$.
\end{proof}

\begin{definition} \label{Def:CategoricalInflation}
    Suppose that $\s{C}$ is a fusion category over $\bb K$, and let $\mathbb E=\Omega{\mathcal C}$. Suppose that $\bb L/\bb E$ is a separable extension.  The categorical inflation of $\s C $ along this extension is
    \[\Infl_{\bb E}^{\bb L}(\s C) := \1_1\otimes(\bb L\Bim_{\bb K}\boxtimes\s C)\otimes\1_1\;,\]
    following the notation of Lemma \ref{Lem:MoritaEquivalence}.
\end{definition}

Note that $\Infl_{\bb E}^{\bb L}(\s C)$ is still only $\bb K$-linear, so this is an operation on fusion categories over a fixed field $\bb K$ and not a kind of base extension.

\begin{corollary}\label{Cor:PropertiesOfInfl}
    The endomorphism field of $\Infl_{\bb E}^{\bb L}(\s C)$ is $\bb L$, and the category $\s M(\bb L/\bb E)$ is a Morita equivalence between $\s C$ and $\Infl_{\bb E}^{\bb L}(\s C)$.
\end{corollary}

\subsection{Categorical and cohomological inflation}

What happens to categories of the form $\Vec_{\mathbb E}^\omega\big(\Gal(\mathbb E/\mathbb K)\big)$ when we perform the categorical inflation to a larger field?

\begin{theorem}\label{Thm:CategorifiesInflation}
    Suppose that $\bb{L}\supseteq\bb{E}\supseteq\bb{K}$ is a tower of Galois extensions, then 
    \[ \Infl^{\bb L}_{\bb E}\Big(\Vec_{\mathbb E}^\omega\big(\Gal(\mathbb E/\mathbb K)\big)\Big) \mathop{\simeq}\limits^{\otimes} \Vec_{\mathbb L}^{\infl_{\bb E}^{\bb L}(\omega)}\left(\Gal(\mathbb L/\mathbb K)\right)\,,\]
    where $\infl_{\bb E}^{\bb L}(\omega)$ is cohomological inflation as in Definition \ref{Def:ClassicalInflation}.
\end{theorem}

\begin{proof}
    Set $\s C=\Vec_{\mathbb E}^\omega\big(\Gal(\mathbb E/\mathbb K)\big)$, and let $G=\Gal(\mathbb E/\mathbb K)$ and $\Gamma=\Gal(\mathbb L/\mathbb K)$, so that there is a surjective homomorphism $q:\Gamma\twoheadrightarrow G$.
    
    The simple objects of $\Infl_{\bb E}^{\bb L}(\s C)$ are of the form $\1_1\otimes(\mathbb L_{\gamma}\boxtimes \bb E_g)\otimes\1_1$, where $\gamma\in\Gamma$, and $\1_1$ is the summand of $\1$ corresponding to embedding $\mathbb E\subseteq\mathbb L$ as in Lemma \ref{Lem:IndecSummands}.

    Similarly to the proof of Proposition \ref{Prop:GrpHomomorphism}, for any nonzero $e\in\bb E$, we can compute that

    \begin{align*}
        &\id_{\1_1}\otimes\big(\id_{\bb L_\gamma}\boxtimes\id_{\bb E_g}\big)\otimes\id_{\1_1}\\
        &=\id_{\1_1}\otimes\big(\id_{\bb L_\gamma}\boxtimes\id_{\bb E_g}\cdot ee^{-1}\big)\otimes\id_{\1_1}\\
        &=\id_{\1_1}\otimes\big(\id_{\bb L_\gamma}\cdot e\boxtimes\id_{\bb E_g}\cdot e^{-1}\big)\otimes\id_{\1_1}\\
        &=\id_{\1_1}\otimes\big(\gamma(e)\cdot\id_{\bb L_\gamma}\boxtimes\id_{\bb E_g}\cdot e^{-1}\big)\otimes\id_{\1_1}\\
        &=\id_{\1_1}\otimes\big(\id_{\bb L_\gamma}\boxtimes\gamma(e)\cdot\id_{\bb E_g}\cdot e^{-1}\big)\otimes\id_{\1_1}\\
        &=\id_{\1_1}\otimes\big(\id_{\bb L_\gamma}\boxtimes\id_{\bb E_g}\cdot g^{-1}\gamma(e)e^{-1}\big)\otimes\id_{\1_1}\,.
    \end{align*}

    A visual representation of this computation in figure \ref{fig:1}.

    \begin{figure}
        \centering
        \includegraphics[width=\textwidth]{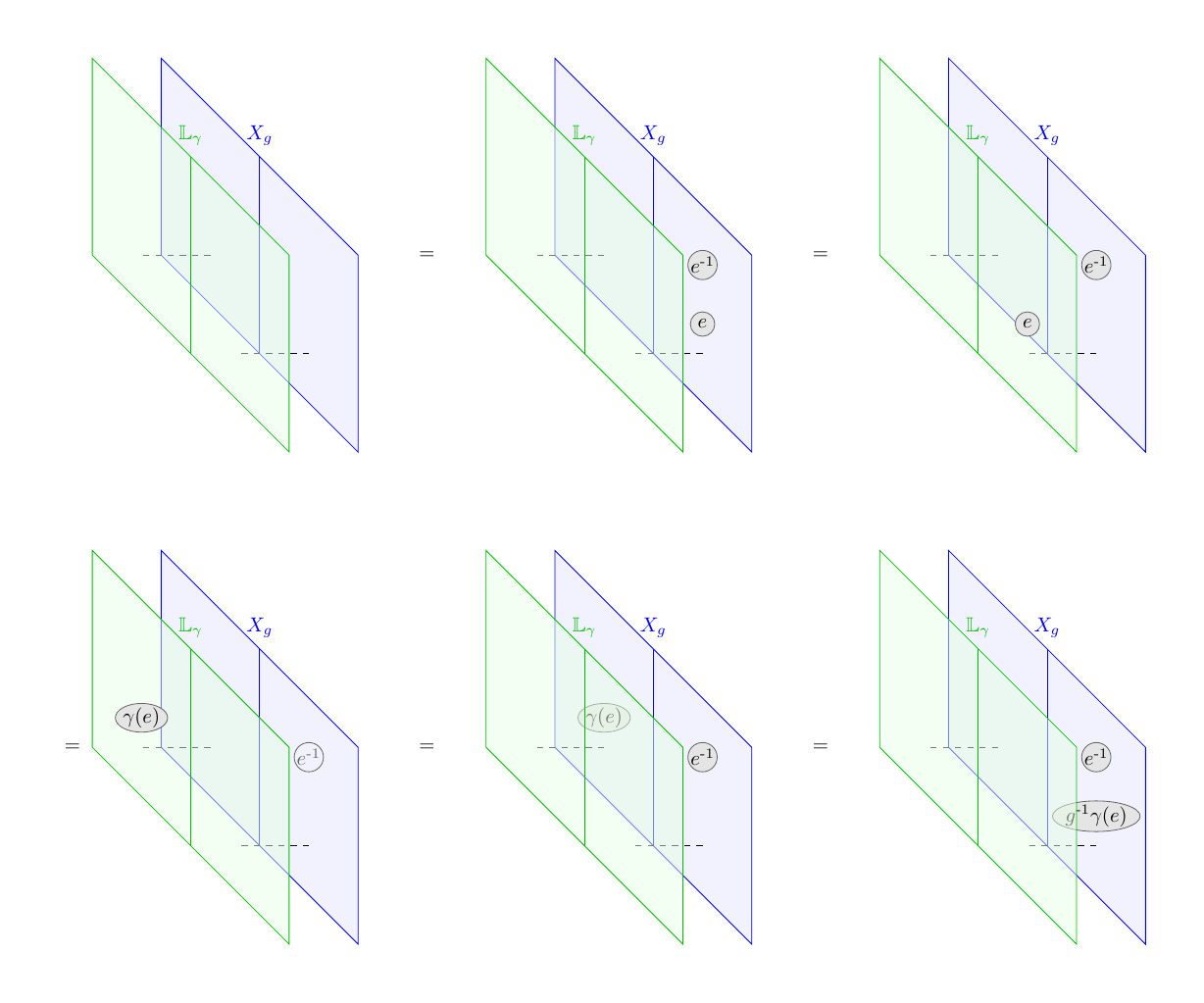}
        \caption{Diagrammatic description of the proof of \ref{Thm:CategorifiesInflation}}
        \label{fig:1}
    \end{figure}
    
    Observe that if $\gamma|_{\mathbb E}\neq g$, then there is at least one morphism $e\in\mathbb E$ such that $b=g^{-1}\gamma(e)e^{-1}\neq1$.  The above computation shows that $\id_{\1_1}\otimes(\id_{\mathbb L_{\gamma}}\boxtimes\id_{\bb E_g}\cdot(b-1))\otimes\id_{\1_1}=0$, and this forces $\1_1\otimes(\mathbb L_{\gamma}\boxtimes \bb E_g)\otimes\1_1=0$.  Thus, the only objects that survive compression by $\1_1$ are those $\mathbb L_\gamma\boxtimes \bb E_g$ for which $q(\gamma)=\gamma|_{\bb E}=g$, so we find that the simple objects in $\Infl_{\bb E}^{\bb L}(\s C)$ are indexed by $\Gamma$.

    The simple objects all have $\bb L$ as their endomorphism algebra, and their fusion rules match that of $\bb L\Bim_{\bb K}$, so it remains to compute the associator.  Let $\omega\in Z^3(G;\mathbb E^\times)$ be the cocycle determined by the associator in $\mathcal C$.  Since the only interesting scalar in the associator appears on the right $\boxtimes$-factor, the cocycle $\omega'\in Z^3(\Gamma;\bb L^\times)$ can be computed directly to be
    \[\omega'(\gamma_1,\gamma_2,\gamma_3)=\omega\big(\gamma_1|_{\bb E},\gamma_2|_{\bb E},\gamma_3|_{\bb E}\big)\,.\]
    This construction is precisely the cocycle description of cohomological inflation (see Definition \ref{Def:ClassicalInflation}), so $\omega'=\infl_{\bb E}^{\bb L}(\omega)$ as desired.
\end{proof}

\begin{corollary}\label{cor:well-defined}
    The collection of group homomorphisms $$\Big\{\psi_{\bb L}:H^3\big(\Gal(\bb L/\bb K);\bb L^\times) \rightarrow \Inv(\bb K)\Big\}_{\bb L/\bb K\,\text{ Galois}}$$ descends to map out of the colimit $\Psi: H^3(\bb K;\bb G_m) \rightarrow \Inv(\bb K)$.
\end{corollary}

\begin{proof}
    Let $\omega\in H^3\big(\Gal(\bb E/\bb K);\bb E^\times\big)$, and suppose $\bb L\supseteq\bb E\supseteq\bb K$ is a tower of Galois extensions.  Our results so far combine to give us the following sequence of equalities inside of $\Inv(\bb K)$
    \begin{align*}
        \psi_{\bb L}\big(\infl_{\bb E}^{\bb L}(\omega)\big)&:=\Big[\psi_{\bb L}^0\big(\infl_{\bb E}^{\bb L}(\omega)\big)\Big]_\sim\\
        &\mathop{=}\limits^{\ref{Thm:CategorifiesInflation}}\Big[\Infl_{\bb E}^{\bb L}\big(\psi_{\bb E}^0(\omega)\big)\Big]_\sim\\
        &\mathop{=}\limits^{\ref{Cor:PropertiesOfInfl}}\big[\psi_{\bb E}^0(\omega)\big]_\sim\;=\;\psi_{\bb E}(\omega)\,.
    \end{align*}
    This can be interpreted as saying that the collection $\{\psi_{\bb L}\}$ forms a cocone under the directed system in Proposition \ref{Prop:AbsoluteCohomologyIsAColimit}, and so the universal property of the colimit provides the assembled map.
\end{proof}

\section{Splitting fields} \label{Sec:Splitting}

\begin{definition}\label{Def:SplittingFieldOfACategory}
    Suppose that $\mathcal C$ is a multi-fusion category over $\bb K$.  We call a Galois extension $\bb L/\bb K$ a splitting field for $\mathcal C$ if it splits all the division algebras $\End(X)$ in the sense that $\bb L\otimes_{\bb K}\End(X)$ is a product of matrix algebras \emph{over $\bb L$}.
\end{definition}

Note that if $\bb L\otimes_{\bb K}\End(X)$ is a product of matrix algebras over field extensions of $\bb L$ we do not consider that a splitting field. For example, the splitting field of $\Vec_{\bb C}$ (thought of as a fusion category over $\bb R$) is $\mathbb{C}$ because $\bb R\otimes_{\bb R}\mathbb{C}$ is not a matrix algebra over $\mathbb{R}$, while $\mathbb{C} \otimes_{\bb R} \mathbb{C} \cong \mathbb{C} \oplus \mathbb{C}$ \emph{is} a product of matrix algebras over $\mathbb{C}$.

\begin{lemma}\label{Lem:SplittingFieldsExist}
    Every multi-fusion category $\s C$ over $\bb K$ admits a splitting field.
\end{lemma}

\begin{proof}
    Since $\s C$ is multi-fusion, for every simple object $X$, the division algebra $\End(X)$ and hence also the field $Z\big(\End(X)\big)$ are separable algebras over $\bb K$ (See Definition \ref{Def:multi-fusion}).  For each simple object $X$, let $\bb F_X$ be a splitting field for $\End(X)$ in the sense that
    \[\bb F_X\otimes_{Z(\End(X))}\End(X)\cong M_{n_X}(\bb F_X)\,.\]
    
    Taking the tensor product rel $\bb K$ of all the $\bb F_X$, we can produce the algebra
    \[A\,=\bigotimes_{X\in\mathrm{Irr}(\s C)}\bb F_X\,.\]
    This algebra is necessarily a (cartesian) product of separable field extensions of every $\bb F_X$.  Let $\bb L_0$ be any of these field factors of $A$, and define $\bb L$ to be the normal closure of $\bb L_0/\bb K$.  It follows that $\bb L$ contains the normal closure of  $Z\big(\End(X)\big)/\bb K$ for every simple object $X$, and will also split every $\End(X)$.

    We can now compute that
    \begin{align*}
        \mathbb L\otimes_{\mathbb K}\End(X)&\cong\mathbb L\otimes_{\mathbb K}Z\big(\End(X)\big)\otimes_{Z(\End(X))}\End(X)\\
        &\cong\left(\prod_{i=1}^m\mathbb L\right)\otimes_{Z(\End(X))}\End(X)\\
        &\cong\prod_{i=1}^m M_{n_i}(\mathbb L).
    \end{align*}
    Since this works for every simple object $X$, $\bb L$ is a splitting field for $\s C$.
\end{proof}

Now suppose that $\mathcal C$ is invertible.  Using categorical inflation with respect to a Galois extension $\bb L/\bb K$ allows us to alter the category without altering the Morita equivalence class.  However if $\bb L$ is generic, there is no guarantee that the inflated category will be easier to work with than $\s C$ itself.  The following theorem combines the running themes of this section by considering inflation when $\bb L$ is a splitting field for an invertible category $\s C$.

\begin{theorem}\label{Thm:InflationAlongASplittingField}
    Let $\s C$ be an invertible fusion category over $\bb K$ with $\Omega\s C=\bb E$.  If $\bb L$ is a splitting field for $\s C$, then
    \[\Infl_{\bb E}^{\bb L}(\s C)\mathop{\simeq}\limits^{\otimes}\Vec_{\bb L}^{\omega}\big(\Gal(\bb L/\bb K)\big)\;,\]
    for some Galois cocycle $\omega$.   In particular, any invertible $\s C$ is Morita equivalent to a category of the form $\Vec_{\mathbb L}^\omega\big(\Gal(\mathbb L/\mathbb K)\big)$.
\end{theorem}

\begin{proof}
    Let $\s C$ be as above, and let $\bb L$ be a splitting field for $\s C$, which exists by Lemma \ref{Lem:SplittingFieldsExist}.  Let us denote $\s D=\Infl_{\bb E}^{\bb L}(\s C)$. Our goal is to show that $\s D$ is faithfully Galois-graded and pointed, with a unique simple object in each grading.  Taken together, these properties imply that $\mathcal D\simeq\Vec_{\mathbb L}^\omega\big(\Gal(\mathbb L/\mathbb K)\big)$ for some cocycle $\omega\in Z^3\big(\Gal(\bb L/\bb K);\bb L^\times\big)$.

    First we show that $\s D$ is Galois graded using Proposition \ref{Prop:GaloisGrading}, namely we want to compute the endomorphism algebras of each simple.
    All of the simple objects in $\bb L\Bim_{\bb K}$ are of the form $\bb L_g$ for some $g\in\Gal(\bb L/\bb K)$ (See Example \ref{Eg:LBimK}).  The endomorphism algebras of tensors of two simple objects must be
    \[\End(\mathbb L_g\boxtimes X)\;\cong\;\mathbb L\otimes_{\mathbb K}\End(X)\;\cong\;\prod_{i=1}^m M_{n_i}(\mathbb L)\,.\]
    This computation shows that the object $\mathbb L_g\boxtimes X$ of the multi-fusion category $\mathbb L\Bim_{\bb K}\boxtimes\mathcal C$ decomposes into m distinct simple subobjects, with the $i^\text{th}$ simple object occurring with multiplicity $n_i$.  More importantly, this also shows that each of these simple summands has endomorphism algebra $\mathbb L$.  Since $\s D$ is a subcategory of $\mathbb L\Bim\boxtimes\mathcal C$, this property holds for $\s D$ as well. Since all the endomorphism algebras of $\s D$ are $\bb L$, it follows from Proposition \ref{Prop:GaloisGrading} that $\s D$ is Galois graded. 
    
    We now aim to show that this grading is \emph{faithful} (\emph{i.e.} $\mathcal D_{g}\neq0$ for every $g\in\Gal(\mathbb L/\mathbb K)$, see Subsection \ref{Sec:GradedFusion}) by considering the Drinfeld center. Since $\s D$ and $\s C$ are Morita equivalent and $\s C$ is invertible, it follows that $\s D$ is also invertible.  By Theorem \ref{Thm:InvertibleIFFZ(C)=Vec} we find that $\s Z(\s D)\simeq\Vec_{\mathbb K}$ .  By \cite[Prop 4.1.5]{sanfordThesis}, the endomorphism field $\Omega\mathcal Z(\mathcal D)$ is given by
    \[\End(\1_{\mathcal Z(\mathcal D)})=\big\{l\in \mathbb L~|~\forall X\in\mathcal D \text{ simple, }l\cdot\id_X=\id_X\cdot l\big\}\,.\]
    Thus $\Omega\s Z(\s C)=\mathbb K$ must be the fixed field of the subgroup of $\Gal(\mathbb L/\mathbb K)$ generated by all those $g$ for which $\mathcal D_g\neq0$.  Since $\mathbb K$ is clearly the fixed field of the whole group, we conclude that the $g$ for which $\mathcal D_g\neq0$ must generate all of $\Gal(\mathbb L/\mathbb K)$.  Since $\mathcal D$ is fusion and the grading is monoidal, we can conclude that the grading is faithful.

    Finally we analyze each of the graded components and show that they each have a single simple object which is invertible.
    
    For any fusion category $\mathcal D$ that is faithfully graded by a group $G$, Proposition \ref{Prop:GradedComponentsInvertible} tells us that each of the components $\mathcal D_g$ is an invertible bimodule category for $\mathcal D_1$. 
    
    The forgetful functor $F:\mathcal \Vec_{\bb K} \simeq Z(\mathcal{D})\to\mathcal D$ is monoidal, and so $\1_{\mathcal D}$ must be the only simple object in the essential image of $F$, since the source category has rank 1.  By \cite[Thm 4.1.4]{sanfordThesis} it follows that $\mathcal D_1\simeq\Vec_{\mathbb L}$.
    
    By combining these observations, we find that for each $g\in\Gal(\mathbb L/\mathbb K)$, the graded component $\mathcal D_g$ is a $\mathbb K$-linear invertible bimodule category for $\Vec_{\mathbb L}$.  Since invertible bimodules between fusion categories are indecomposable as modules, it follows from Example \ref{Eg:IndecVecModules} that all invertible bimodule categories have rank 1.  Thus we find that each $\mathcal D_g$ has a unique simple object, say $X_g$.

    Since $\1$ is the only simple object in $\mathcal D_1$, we have that $X_g^* \otimes X_g \cong \1^{\oplus d}$, but since $\mathbb L\cong\Hom(X_g,X_g)\cong\Hom(X_g^*\otimes X_g,\1)$ we have that $d=1$. Hence $X_g$ is an invertible object for every $g$, and so $\mathcal D$ is a pointed fusion category.
\end{proof}

\section{Obstruction to Morita triviality} \label{sec:Obstruction}

We wish to determine when invertible categories are Morita trivial.  By Theorem \ref{Thm:InflationAlongASplittingField}, this is equivalent to determining when categories of the form $\mathcal C=\Vec_{\bb L}^\omega\big(\Gal(\bb L/\mathbb K)\big)$ are Morita trivial.  We have seen in Theorem \ref{Thm:ExamplesAreInvertible} that a sufficient condition is that $\omega=1$, but the exact condition is more subtle.

We begin with a definition first considered by Teichmuller in \cite{teichmuller}.

\begin{definition}
    Let $G=\Gal(\bb L/\mathbb K)$ be the Galois group of a Galois extension, and let $\mathbb D$ be an $\bb L$-central simple algebra.  $\mathbb D$ is said to be $G$-normal if for every $g\in G$, there exists a $\mathbb K$-algebra isomorphism $\Tilde{g}:\mathbb D\to\mathbb D$ such that $\Tilde{g}|_{\bb L}=g$.
\end{definition}

Let us write $\bb L_g$ for the $\bb L$-bimodule whose right action is twisted by $g\in\Gal(\bb L/\bb K)$.

\begin{proposition}\label{Prop:G-ActsOnBr(L)}
    Let $G=\Gal(\bb L/\mathbb K)$ be the Galois group of a Galois extension, and let $\mathbb D$ be an $\bb L$-central simple algebra.  The algebra $\bb L_g\otimes_{\bb L}\bb D$ is also an $\bb L$-central simple algebra, and the assignment $g_*[\bb D]:=[\bb L_g\otimes_{\bb L}\bb D]$ defines an action of $\Gal(\bb L/\bb K)$ on $\Br(\bb L)$.
\end{proposition}

It is clear that every $G$-normal algebra $\mathbb D$ gives rise to a fixed point $[\mathbb D]\in\Br(\bb L)^G$, but in fact a strong converse is also true.

\begin{lemma}\label{Lem:GNormalAlgebras=FixedPoints}
    Every fixed point $X$ of the action of $G$ on $\Br(\bb L)$ is of the form $X=[\mathbb D]$, for some $G$-normal division algebra $\mathbb D$.
\end{lemma}

\begin{proof}
    If $X\in\Br(\bb L)^G$, then it is of the form $X=[\mathbb D]$ for some division algebra $\mathbb D$.  For any $g\in G$, we have that 
    \[[\bb L_{g^{-1}}\otimes_{\bb L}\mathbb D]=:(g^{-1})_*X=X=[\mathbb D],\]
    and so we know that $\bb L_{g^{-1}}\otimes_{\bb L}\mathbb D$ and $\mathbb D$ are Morita equivalent.  Since they are both division algebras, it follows that they are actually isomorphic as $\bb L$ algebras, but this is the same as having a $g$-linear isomorphism from $\mathbb D$ to itself.
\end{proof}

In light of Lemma \ref{Lem:GNormalAlgebras=FixedPoints}, we will write $\Br(\bb L)^G$ for the collection of $G$-normal division algebras up to isomorphism.  With our notation established, we are finally ready to state our first important observation of this section.

\begin{proposition}\label{Prop:MoritaTrivialsLookLikeDBim}
    Suppose that $\mathcal C=\Vec_{\bb L}^\omega\big(\Gal(\bb L/\mathbb K)\big)$ is Morita equivalent to $\Vec_{\mathbb K}$.  The underlying category of any such Morita equivalence is of the form $\mathbb D\Mod$, for some $G$-normal division algebra $\mathbb D$.
\end{proposition}

\begin{proof}
Suppose that there is a Morita equivalence $\mathcal M$ from $\mathcal C$ to $\Vec_{\mathbb K}$.  Since $\mathcal C \simeq \End_\mathcal{\Vec_{\mathbb K}}(\mathcal{M})$ is fusion it follows that $\mathcal M$ must only have one simple object $m$. Set $\mathbb D=\End(m)$, then it follows that $\bb D$ is a division algebra and $\mathcal M\simeq\mathbb D\Mod$ as module categories over $\Vec_{\mathbb K}$.

Since $\mathcal M$ is a Morita equivalence, $\mathcal C\simeq\bb D\Bim_{\bb K}$.  Let us analyze the structure of the category $\bb D\Bim_{\bb K}$.  The monoidal unit of $\bb D\Bim_{\bb K}$ is the algebra $\mathbb D$ itself, thought of as a $\mathbb D$-bimodule.  The endomorphisms of $\1=\mathbb D$ are $\End(\mathbb D)=Z(\mathbb D)$, and since $\mathcal C\simeq\bb D\Bim_{\bb K}$, we find that $Z(\mathbb D)=\bb L$.  We can compute that, as $\mathbb K$-algebras,
\begin{align*}
    \mathbb D\otimes_{\mathbb K}\mathbb D^{op}&\cong\mathbb D\otimes_{\bb L}\big(\bb L\otimes_{\mathbb K}\bb L\big)\otimes_{\bb L}\mathbb D^{op}\\
    &\cong\mathbb D\otimes_{\bb L}\left(\bigoplus_{g\in\Gal(\bb L/\mathbb K)}\bb L_g\right)\otimes_{\bb L}\mathbb D^{op}\\
    &\cong\bigoplus_{g\in\Gal(\bb L/\mathbb K)}\mathbb D\otimes_{\bb L}\bb L_g\otimes_{\bb L}\mathbb D^{op}\,.
\end{align*}
The algebras $\mathbb D\otimes_{\bb L}\bb L_g\otimes_{\bb L}\mathbb D^{op}$ are central simple over $\bb L$, and their Brauer class is $[\mathbb D]\cdot g_*[\mathbb D]^{-1}$.

By looking at our original category $\mathcal C$, we find that every simple object has $\bb L$ as its endomorphism algebra.  Thus it follows that each of the classes $[\mathbb D]\cdot g_*[\mathbb D]^{-1}$ must be trivial in $\Br(\bb L)$, and so $\bb D$ is $G$-normal by Lemma \ref{Lem:GNormalAlgebras=FixedPoints}.\qedhere

\end{proof}

In \cite{teichmuller}, Teichm\"uller associated a Galois cohomology class $T(\mathbb D)\in H^3(G;\bb L^\times)$ to each $G$-normal algebra $\mathbb D$.  In the article \cite{MR25443}, Eilenberg and Mac Lane established that $T:\Br(\bb L)^G\to H^3(G;\bb L^\times)$ is a homomorphism, and they also completely characterized the image of $T$.

\begin{theorem}[{\cite[cf. Thm 7.1]{MR25443}}]\label{Thm:E&MLClassification}
    The image $\im(T)\leq H^3(G;\bb L^\times)$ of the Teichm\"uller construction is the subgroup $H^3_0(G;\bb L^\times)$ consisting of all classes that trivialize upon inflation\footnote{Eilenberg and Mac Lane referred to `lifting' the cocycle, but this process coincides with what is now known as inflation.} to some finite Galois extension.
\end{theorem}

The quantification over finite Galois extensions in the above theorem can be avoided by rephrasing this theorem in terms of the absolute Galois group $G_{\mathbb K}=\Gal(\mathbb K^{\text{sep}}/\mathbb K)$.

\begin{corollary}[Alternate version of Theorem \ref{Thm:E&MLClassification}]\label{Cor:E&MLClassification}
    For any Galois extension $\bb L/\mathbb K$, 
    \[\im(T)=\ker\Big(\infl:H^3(G;\bb L^\times)\to H^3\big(\bb K;\bb G_m\big)\Big)\,.\]
\end{corollary}

\subsection{G-normal algebras and Morita triviality}
 
Let us now spend some time explaining how to interpret $T(\mathbb D)$ categorically.

\begin{theorem}\label{Thm:MoritaInterpOfT}
    Set $G=\Gal(\bb L/\bb K)$, and let $\bb D$ be a $G$-normal algebra.  The category $\bb D\Bim_{\bb K}$ of $\bb K$-linear $\bb D$-bimodules is equivalent to a category of Galois graded vector spaces with associator cocycle $T(\bb D)$,
    \[\text{i.e.}\hspace{1cm}\bb D\Bim_{\bb K}\;\simeq\;\Vec_{\bb L}^{T(\bb D)}\big(\Gal(\bb L/\bb K)\big)\,.\]
\end{theorem}

\begin{proof}

Let $\mathbb D$ be a $G$-normal algebra, and for every $a\in G$, let $\Tilde{a}\in\Aut(\mathbb D)$ be a lifting of $a$.  Using these automorphisms, we define a collection of $\mathbb K$-linear maps
\begin{align*}
 \gamma_a:\bb L_a\otimes_{\bb L}\mathbb D&\longrightarrow\mathbb D\otimes_{\bb L}\bb L_a\\
 l\otimes d&\longmapsto \Tilde{a}(d)\otimes l\;.
\end{align*}

Let $l_a\in\bb L_a$ be the image of $1\in\bb L$ under some isomorphism of left $\bb L$-modules, and define a collection of isomorphisms $f_{a,b}:\bb L_a\otimes_{\bb L}\bb L_b\to\bb L_{ab}$ by the formula $f_{a,b}(l_a\otimes l_b)=l_{ab}$.  Since $\bb L$ is Galois over $\mathbb K$, and $\mathbb D$ is central simple over $\bb L$, it follows that $\mathbb D$ is separable as a $\mathbb K$-algebra.  Let
\[\theta(d)=\sum_id\cdot u_i\otimes_{\mathbb K}v_i=\sum_iu_i\otimes_{\mathbb K}v_i\cdot d\]
be a separator for $\mathbb D$.

We define an $\bb L$-linear projection $p:\mathbb D\to\bb L$ by the formula
\[p(d)=\sum_iu_idv_i\,.\]

The simple objects in the category of $\mathbb D$-bimodules are all of the form $X_a:=\mathbb D\otimes_{\bb L}\bb L_a$ for $a\in G$.  The left $\mathbb D$-module action $\mathbb D\otimes_{\mathbb K}\mathbb D\otimes_{\bb L}\bb L_a\to \mathbb D\otimes_{\bb L}\bb L_a$ on $X_a$ is simply the product $\mu:\mathbb D\otimes\mathbb D\to\mathbb D$ on the $\mathbb D$ factors.  The right module action is given by the composition
\[\mathbb D\otimes_{\bb L}\bb L_a\otimes_{\mathbb K}\mathbb D\twoheadrightarrow\mathbb D\otimes_{\bb L}\bb L_a\otimes_{\bb L}\mathbb D\mathop{\longrightarrow}\limits^{\id\otimes\gamma_a}\mathbb D\otimes_{\bb L}\mathbb D\otimes_{\bb L}\bb L_a\mathop{\longrightarrow}\limits^{\mu\otimes\id}\mathbb D\otimes_{\bb L}\bb L_a\,.\]
That this is a valid module action follows from the fact that $\Tilde{a}$ is an automorphism.

In order to investigate the associator in this category, it is necessary to choose a collection of bimodule isomorphisms $X_a\otimes_{\mathbb D}X_b\to X_{ab}$.  There is an obvious $\bb L$-linear isomorphism given by
\[(\mu\otimes f_{a,b})\circ(\id\otimes\gamma_a\otimes \id)\,.\]
This na\"ive map fails to be a morphism of right modules.  To see this, let's drop the tensor product symbols, and compute what this map does on elements.  The above morphism corresponds to the following formula
\[d_1l_ad_2l_b\mapsto d_1\Tilde{a}(d_2)l_{ab}\,.\]
Pre-multiplication by $d_3$ on the right produces
\begin{align*}
 d_1l_ad_2l_bd_3&\mapsto d_1l_ad_2\Tilde{b}(d_3)l_b\\
 &\mapsto d_1\Tilde{a}\big(d_2\Tilde{b}(d_3)\big)l_al_b\\
 &= d_1\Tilde{a}(d_2)\Tilde{a}\big(\Tilde{b}(d_3)\big)l_{ab}\,.\\
\end{align*}
Post-multiplication by $d_3$ on the right produces
\begin{align*}
 d_1l_ad_2l_bd_3&\mapsto d_1\Tilde{a}(d_2)l_{ab}d_3\\
 &\mapsto d_1\Tilde{a}(d_2)\widetilde{\,ab\,}(d_3)l_{ab}\,.
\end{align*}
The discrepancy between the two morphisms is the automorphism $\Tilde{a}\Tilde{b}\widetilde{\,ab\,}^{-1}$ on the $d_3$ factor.  By the Skolem-Noether theorem, this automorphism is inner, and so there is some nonzero $r(a,b)\in\mathbb D$ with the property that
\[r(a,b)\cdot\widetilde{\,ab\,}(d)=\Tilde{a}\Tilde{b}(d)\cdot r(a,b)\,,\]
for every $d\in\mathbb D$.  Using the central projection $p:\mathbb D\to\bb L$, we can assume without loss of generality that $p\big(r(a,b)\big)=1$.  This allows us to use multiplication by $r(a,b)$ to correct the na\"ive morphism, and turn it into a morphism of bimodules as follows
\[m_{a,b}:=\big(\mu\big(-\otimes r(a,b)\big)\otimes \id\big)\circ(\mu\otimes f_{a,b})\circ(\id\otimes\gamma_a\otimes \id)\,.\]
The associator coefficients for this category $\bb D\Bim_{\bb K}\simeq\Vec_{\bb L}^\omega(G)$ are a collection of scalars $\{\omega(a,b,c)\in\bb L^\times~|~a,b,c\in G\}$.  Since the associator is inherited from $\Vec_{\mathbb K}$, these scalars $\omega(a,b,c)$ can be computed using the following formula
\[\omega(a,b,c)\cdot\id_{X_{abc}}\;=\;m_{a,bc}\circ(\id\otimes m_{b,c})\circ(m_{a,b}^{-1}\otimes\id)\circ m_{ab,c}^{-1}\,.\]
A lengthy yet elementary computation can now be made by appealing to our definition of $m_{a,b}$, together with our assumption that $p\big(r(a,b)\big)=1$.  On elements, this yields the formula
\begin{align*}
\omega(a,b,c)\cdot d\otimes l_{abc}\;&=\;\big(d\cdot r(ab,c)^{-1}r(a,b)^{-1}\Tilde{a}\big(r(b,c)\big)r(a,bc)\big)\otimes l_{abc}\\
&=\;\big(d\cdot \delta(r)(a,b,c)\big)\otimes l_{abc}\\
&=\;\delta(r)(a,b,c)\cdot d\otimes l_{abc}\,.
\end{align*}
The coboundary $\delta(r)(a,b,c)$ is precisely the definition of $T(\mathbb D)$ as outlined in \cite[Eqn 6.3]{MR25443}.  
\end{proof}

\begin{theorem}\label{Thm:TrivialityCondition}
    A category of the form $\Vec_{\bb L}^\omega\big(\Gal(\bb L/\bb K)\big)$ is Morita trivial if and only if $\infl(\omega)=1\in H^3(\bb K;\bb G_m)$.
\end{theorem}

\begin{proof}
    If $\Vec_{\bb L}^\omega\big(\Gal(\bb L/\bb K)\big)$ is Morita trivial, Proposition \ref{Prop:MoritaTrivialsLookLikeDBim} shows that it must be equivalent to $\bb D\Bim_{\bb K}$ for some $G$-normal algebra $\mathbb D$.  Theorem \ref{Thm:MoritaInterpOfT} establishes that $\omega=T(\mathbb D)$, and the characterization of $\im(T)$ from Corollary \ref{Cor:E&MLClassification} shows us that $\infl(\omega)=1$.

    Conversely, if $\infl(\omega)=1$, then there must be some finite Galois extension $\mathbb F/\bb L$ such that $\infl_{\bb L}^{\bb F}(\omega)=1\in H^3\big(\Gal(\mathbb F/\mathbb K);\mathbb F^\times\big)$.  Thus by using categorical inflation (Proposition \ref{Thm:CategorifiesInflation}), we find that the category $\Vec_{\bb L}^\omega(G)$ is Morita equivalent to 
    $$\Vec^{\infl_{\bb L}^{\bb F}(\omega)}_{\mathbb F}\big(\Gal(\mathbb F/\mathbb K)\big)\;=\;\Vec^{1}_{\mathbb F}\big(\Gal(\mathbb F/\mathbb K)\big)\;=\;\mathbb F\Bim_{\bb K}\,.$$
    This last category is evidently Morita trivial, so we are done.
\end{proof}

We are now equipped to prove the main theorem of the entire paper.

\begin{theorem}\label{Thm:H3Classification}
    The map $\Psi: H^3(\bb K;\bb G_m) \rightarrow \Inv(\bb K)$ is an isomorphism of groups.
\end{theorem}

\begin{proof}
    Surjectivity of $\Psi$ follows immediately from Theorem \ref{Thm:InflationAlongASplittingField}, so we need only show injectivity.
    
    Proposition \ref{Prop:AbsoluteCohomologyIsAColimit} shows that every class $\beta\in H^3(\bb K;\bb G_m)$ is of the form $\infl(\omega)$ for some finite field extension $\bb L/\bb K$, and some $\omega\in H^3\big(\Gal(\bb L/\bb K);\bb L^\times)$.

    Suppose $\beta=\infl(\omega)$ is in $\ker(\Psi)$.  Then we would have
    \[[\Vec_{\bb K}]_{\sim}\;=\;\Psi(\beta)\;=\;\psi_{\bb L}(\omega)\;=\;\Big[\Vec_{\bb L}^\omega\big(\Gal(\bb L/\bb K)\big)\Big]_\sim\,.\]
    Thus by Theorem \ref{Thm:TrivialityCondition}, $1=\infl(\omega)=\beta$, so $\Psi$ is injective.
\end{proof}

\section{The Atiyah-Hirzebruch spectral sequence} \label{Sec:spectral}

Up to this point we have presented a direct technical computation of $\Inv(\bb K)$.  Here we present an alternative high-level argument that establishes the isomorphism $H^3(\bb K;\bb G_m) \rightarrow \Inv(\bb K)$ abstractly using a spectral sequence. This abstract argument does not establish that each equivalence class is represented by a fusion category (let alone one of a specific form).

Given a field $\bb K$, we can form the 3-category $\mFus_{\bb K}$ (as outlined in \cite{johnson-freydOpLaxNaturalTransformations}, see also \cite{douglasDualizable}), whose objects are multi-fusion categories over $\bb K$, 1-morphisms from $\s C$ to $\s D$ are locally separable (over $\bb K$) $\s C$-$\s D$ bimodule categories, 2-morphisms are $\bb K$-linear bimodule functors, and 3-morphisms are bimodule natural transformations.  This category admits a symmetric monoidal structure given by the Deligne product $\boxtimes=\boxtimes_{\bb K}$.

The core of $\mFus_{\bb K}$ is the symmetric monoidal 3-groupoid $\underline{\Inv_{\bb K}}$ obtained by removing all non-invertible morphisms at every level.  The groupoid $\underline{\Inv_{\bb K}}$ can be thought of as a homotopy 3-type via the homotopy hypothesis, or equivalently as a spectrum.  We wish to apply the Atiyah-Hirzebruch spectral sequence to (Galois-)twisted-equivariant generalized cohomology with coefficients in this spectrum, and for this purpose it is convenient to index this spectrum in such a way that it is coconnective.  In other words, our purposes require that top-dimensional morphisms contribute homotopy at level zero.  Thus the homotopy groups of our spectrum $\underline{\Inv_{\bb K}}$ are given as follows
\begin{gather*}
    \pi_k\underline{\Inv_{\bb K}}\;=\;
    \begin{cases}
        \bb K^\times & k=0\\
        0 & k=-1\\
        \Br(\bb K) & k=-2\\
        \Inv(\bb K) & k=-3\\
        0 & \text{all other }k\,.
    \end{cases}
\end{gather*}

Given a separable field extension $\bb L/\bb K$, there is a 3-functor $\mFus_{\bb K}\to\mFus_{\bb L}$ called base extension that is given by everywhere applying $\Vec_{\bb L}\boxtimes_{\bb K}(-)$.  This functor is faithful (on 3-morphisms), and hence restricts to a morphism $\underline{\Inv_{\bb K}}\to\underline{\Inv_{\bb L}}$ of spectra.

When $\bb L/\bb K$ is Galois, the Galois group $G=\Gal(\bb L/\bb K)$ acts on $\underline{\Inv_{\bb L}}$, and the 3-groupoid of homotopy fixed points of this action is canonically equivalent to $\underline{\Inv_{\bb K}}$.  Thus it follows that the coefficient groups
\begin{gather*}
    \underline{\Inv_{\bb K}}^n(\text{pt})\;:=\;H^n(\text{pt};\underline{\Inv_{\bb K}})\;=\;H^n_G(\text{pt};\underline{\Inv_{\bb L}})
\end{gather*}
can be computed via the Atiyah-Hirzebruch spectral sequence
\[E^{p,q}_2\;\cong\;H^p_G(\text{pt};\pi_{-q}\underline{\Inv_{\bb L}})\Rightarrow H^{p+q}_G(\text{pt};\underline{\Inv_{\bb L}})\,.\]
For any $\bb ZG$ module $A$, we can identify the Borel equivariant cohomology $H^*_G(\text{pt};A)\cong H^*(BG;A)$ with (twisted) group cohomology $H^*(G;A)$.  With this identification in place, the $E_2$ page is as follows.

\begin{center}
    \includegraphics[width=\textwidth]{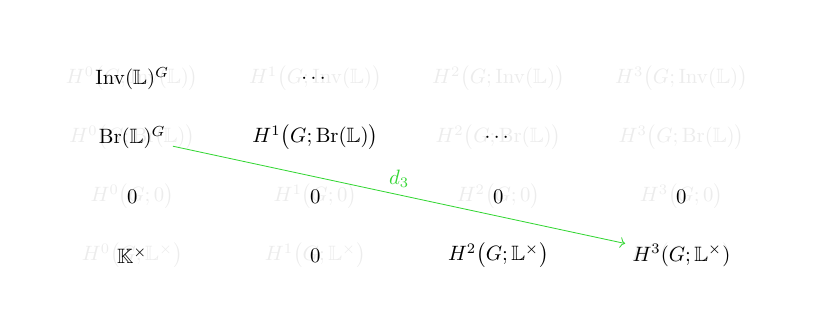}
\end{center}

The terms in total degree 4 that survive to the $E_{\infty}$ page will combine to produce the associated graded group of $\Inv(\bb K)$ with respect to the filtration induced by the extension $\bb L/\bb K$.  Different choices of $\bb L$ will produce different filtrations on $\Inv(\bb K)$, and we can pass to different filtrations via categorical inflation.  Note that inflation produces a different spectral sequence entirely, and should not be confused with turning the page of the current spectral sequence.

It is instructive to consider the case where $\bb L=\bb K^\text{sep}$.  In this case, $\Br(\bb L)=1$, and Example \ref{Eg:Closed=>Inv(k)=1} shows that $\Inv(\bb L)=1$ also.  Thus the groups $E^{p,q}_2$ vanish for $q\geq1$.
Since this sequence has already collapsed and there is no extension problem to solve, we find that
\begin{gather*}
    \Inv(\bb K)=\underline{\Inv_{\bb K}}^{3}(\text{pt})\;\cong\;H^3(G;\bb L^\times)\;=:\;H^3(\bb K;\bb G_m)\,,\;\text{and}\\
    \Br(\bb K)=\underline{\Inv_{\bb K}}^{2}(\text{pt})\;\cong\;H^2(G;\bb L^\times)\;=:\;H^2(\bb K;\bb G_m)\,.
\end{gather*}

Thus the Atiyah-Hirzebruch spectral sequence associated to these spectra and the field extension $\bb K^\text{sep}/\bb K$ serves as a high-level alternative to the proof of Theorem \ref{Thm:H3Classification}, as well as the $H^2$ classification of $\Br(\bb K)$.

The highlighted $d_3$ differential is precisely the Teichmüller construction that produces a 3-cocycle from a $G$-normal algebra.  This explains why all cocycles of the form $T(\bb D)$ do not contribute to $\Inv(\bb K)$.

The Galois group $G$ acts on the collection of all invertible multi-fusion categories over $\bb L$, and the terms in bidegree $(0,3)$ are the classes in $\Inv(\bb L)$ that are fixed by this action.  In order for these terms to contribute to $\Inv(\bb K)$ they must be in the kernel of both $d_2$ and $d_4$ differentials.  Being in both of these kernels is precisely the obstruction theory data required to ensure that a fixed Morita class can be upgraded to be a homotopy fixed point of the action.  Thus $\ker(d_4)$ is the subgroup of $\Inv(\bb L)$ consisting of the underlying classes of homotopy fixed points, and therefore by Galois descent, $\ker(d_4)$ is the collection of all classes in $\Inv(\bb L)$ that arise from extension of scalars $\Inv(\bb K)\to\Inv(\bb L)$.  In analogy with the Brauer group, the intermediate group $\ker(d_2)$ should be thought of as the collection of `$G$-normal fusion categories'.

The cocycles representing terms in bidegree $(1,2)$ are twisted homomorphisms from $G$ to $\Br(\bb L)$, and these will contribute to $\Inv(\bb K)$ when they are in the kernel of $d_3$.  These twisted homomorphisms can be extracted from invertible fusion categories $\s C$ over $\bb K$ that are faithfully Galois graded.  For such categories, $\s C_1\simeq\Vec_{\bb L}$, and every graded component will be of the form $\s C_g\simeq\bb D_g\Mod$ for some $\bb L$-central division algebra $\bb D_g$.  The resulting twisted homomorphism is $f(g):=[\bb D_g]$.

The terms in total degree 2 of the $E_{\infty}$ page will determine $\Br(\bb K)$ as an extension of the form below
\[H^2(G;\bb L^\times)\hookrightarrow\Br(\bb K)\twoheadrightarrow\ker(d_3)\,.\]
It was proven in \cite[Thm 7.1]{MR25443} that $\ker(d_3)$ is precisely those $G$-normal algebras that arise from scalar extension from $\bb K$ to $\bb L$.  In other words, $H^2(G;\bb L^\times)$ is the kernel of the map $\Br(\bb K)\to\Br(\bb L)$ induced by extension of scalars.  This matches the classical description of the relative Brauer group $\Br(\bb L/\bb K)$ as being isomorphic to $H^2(G;\bb L^\times)$.

\bibliography{main}{}
\bibliographystyle{alpha}

\end{document}